\documentclass[11pt]{amsart}

\usepackage{amsmath,amssymb,amsthm,mathtools}
\usepackage{mathrsfs}
\usepackage[margin=1.15in]{geometry}
\usepackage{enumitem}
\usepackage{hyperref}

\hypersetup{
    colorlinks=true,
    linkcolor=blue,
    citecolor=blue,
    urlcolor=blue,
    pdftitle={
        Countable IET Models and Defect Sets for Interval Translation Maps
    },
    pdfauthor={
        Sergey Kryzhevich, Khosro Tajbakhsh, and Reza Yaghmaeian
    },
    pdfsubject={
        Interval translation maps, finite and countable interval exchange
        models, defect measures, and self-similar infinite-type dynamics
    },
    pdfkeywords={
        interval translation map,
        interval exchange transformation,
        countable IET,
        invariant measure,
        metric equivalence,
        defect measure,
        substitution dynamics,
        graph-directed attractor
    }
}

\theoremstyle{plain}
\newtheorem{theorem}{Theorem}[section]
\newtheorem{lemma}[theorem]{Lemma}
\newtheorem{proposition}[theorem]{Proposition}
\newtheorem{corollary}[theorem]{Corollary}

\theoremstyle{definition}
\newtheorem{definition}[theorem]{Definition}
\newtheorem{question}[theorem]{Question}
\newtheorem{example}[theorem]{Example}

\theoremstyle{remark}
\newtheorem{remark}[theorem]{Remark}

\DeclareMathOperator{\supp}{supp}
\DeclareMathOperator{\Leb}{Leb}

\newcommand{\cP}{\mathcal P}
\newcommand{\cH}{\mathcal H}

\newcommand{\cC}{\mathcal C}

\newcommand{\borel}{\mathcal B}

\title[Countable IET models for interval translation maps]
{Countable IET Models and Defect Sets for Interval Translation Maps}

\author[S. Kryzhevich]{Sergey Kryzhevich}
\address{
Institute of Applied Mathematics,
Faculty of Applied Physics and Mathematics,
Gdańsk University of Technology,
80--233 Gdańsk, Poland
}
\email{serkryzh@pg.edu.pl}

\author[K. Tajbakhsh]{Khosro Tajbakhsh $^{*}$}
\thanks{$^{*}$ Corresponding author}
\address{
Department of Mathematics,
Faculty of Mathematical Sciences,
Tarbiat Modares University,
Tehran 14115--134, Iran
}
\email{khtajbakhsh@modares.ac.ir}

\author[R. Yaghmaeian]{Reza Yaghmaeian}
\address{
Department of Mathematics,
Faculty of Mathematical Sciences,
Tarbiat Modares University,
Tehran 14115--134, Iran
}
\email{reza.yaghmaeian@modares.ac.ir}



\date{\today}

\begin{document}

\begin{abstract}
Interval translation maps are piecewise translations for which the images of distinct
continuity intervals may overlap. We study when their measured dynamics can be represented
by finite or countable interval exchange transformations (IETs).

First, we give a direct entropy-based proof of the previously known fact that an interval
translation map is invertible almost everywhere with respect to every non-atomic invariant
probability measure. The proof is based on a polynomial upper bound for the complexity of the
natural branch coding. We then isolate a general consequence of the theorem of Arnoux,
Ornstein and Weiss: every non-atomic measure-preserving automorphism of a standard
probability space has a countable IET model. Applied to the almost-everywhere invertible
core, this yields an abstract countable IET model for every measured interval translation
map under consideration.

We next study the canonical, order-preserving coordinate defined by the distribution function
of the invariant measure. For each branch we introduce a positive defect measure that records
the excess measure carried by the direct image under the branch translation. Its push-forward
to the distribution coordinate gives a canonical sufficient cut set: away from this set the
induced map is locally a translation. Finite defect support yields a finite IET model, while a
Lebesgue-null canonical defect cut set yields a countable IET model. Under branchwise nonsingularity, the
defect support is the closure of the cuts arising from active gaps of the support, giving an
equivalent geometric description of the finite case. Finally, for a self-similar infinite-type
Bruin--Troubetzkoy map, we compute the canonical defect cut set explicitly. It is countably
infinite and Lebesgue-null, so the ordinary distribution coordinate yields a genuinely
countable, non-finite IET model.
\end{abstract}

\maketitle

\section{Introduction}

Interval exchange transformations (IETs) are bijective piecewise translations of an interval.
Interval translation maps are obtained by dropping the global injectivity requirement, so the
images of different continuity intervals may overlap. They therefore form a natural
non-invertible extension of IETs.

Let \(S\) be an interval translation map and let \(\mu\) be a non-atomic
\(S\)-invariant Borel probability measure. Such measures are available in the setting
considered by Kryzhevich \cite{Kryzhevich2020}. Writing
\[
        J=\supp\mu,
\]
we ask whether the measured system \((J,\mu,S)\) can be represented by an IET.

Foundational studies of interval translation maps and their finite-type and
infinite-type behavior include the works of Boshernitzan--Kornfeld and
Schmeling--Troubetzkoy
\cite{BoshernitzanKornfeld1995,SchmelingTroubetzkoy2000}.
Several later results address representation questions more directly. Pires obtained a
topological semiconjugacy to an IET, possibly with flips, for injective
piecewise continuous maps under additional assumptions \cite{Pires2016}. In
\cite[Theorem~3.10]{Kryzhevich2020}, Kryzhevich asserted that an interval translation
map endowed with a non-atomic invariant measure is metrically equivalent to an interval
exchange map via the ordinary distribution coordinate. The published proof relies on the
direct-image equality established there in Lemma~3.9,
\[
        \mu(SB)=\mu(B),
\]
for an interval \(B\) contained in one continuity branch. Invariance, however, controls
full preimages and does not by itself imply this equality for direct images of arbitrary
geometric branch intervals: such intervals may contain gaps of the support, or more general
null subsets, whose branch images carry positive measure. Thus the issue is with this step in
the published argument; we do not claim here that the metric-equivalence conclusion of
\cite{Kryzhevich2020} is false as an abstract measurable statement. Corollary~\ref{cor:universal-countable-iem}
recovers an abstract countable-IET model without using the distribution coordinate, while
Sections~5--7 give explicit sufficient conditions under which the canonical distribution
coordinate itself yields a finite or countable IET. Almost-everywhere invertibility justifies
direct-image invariance on an invariant full-measure invertible core, but not automatically on
the entire geometric interval.

There are therefore two distinct representation problems. The first is abstract and
measure-theoretic: does there exist some measurable coordinate in which the system becomes an
IET? The second is canonical and geometric: does the ordinary
distribution coordinate of \(\mu\) provide such a model while preserving the order inherited
from the original interval? The answers play complementary roles.

We begin with the abstract problem. Almost-everywhere invertibility for interval translation
maps was proved by Tajbakhsh and Yaghmaeian \cite{TajbakhshYaghmaeian2025}. Here we give a
different, self-contained proof. The natural branch coding has polynomial complexity, and
hence zero entropy. A conditional-entropy argument then shows that the present branch is
determined almost everywhere by the image point, which produces a measurable inverse on a
full-measure invariant core. The proof is consistent with Buzzi's zero-entropy theorem for
piecewise isometries \cite{Buzzi2001}, but does not require an identification of topological
entropy notions.

We next separate the interval-translation input from the general representation theorem.
Using the theorem of Arnoux, Ornstein and Weiss \cite{ArnouxOrnsteinWeiss1985}, we show that
every non-atomic measure-preserving automorphism of a standard probability space is
metrically equivalent to a countable IET. Applying this general
statement to the automorphic core gives an abstract countable IET model for
every measured interval translation map considered here. This conjugacy need not preserve the
original order and need not coincide with the distribution coordinate. The AOW consequence is
included primarily to separate the general measurable representation problem from the
canonical order-preserving problem; the underlying representation theorem is not presented as
new.

The main contributions developed here concern the canonical distribution coordinate. We
introduce positive branch-defect measures and use the supports of their push-forwards to define
a canonical defect cut set. We relate these supports to active gaps, prove finite- and
Lebesgue-null support criteria for finite and countable IET models, and carry out a complete
defect computation for a stationary Bruin--Troubetzkoy example. The entropy argument in
Section~3 is an independent, self-contained proof of the previously known
almost-everywhere-invertibility conclusion; it is included to make the paper logically
self-contained and to isolate the automorphic core used by the AOW theorem.

We then turn to the canonical model. Define
\[
        h(x)=\mu([0,x]).
\]
The map \(h\) sends \(\mu\) to Lebesgue measure and becomes a measure-space isomorphism
after suitable null sets are removed. In the \(h\)-coordinate, however, the displacement of a
branch need not remain constant. To measure this failure, we introduce on each branch a positive
defect measure recording the excess measure carried by the translated image. Its push-forward
under \(h\) has a support which forms a canonical sufficient cut set. Outside this set the induced
map is locally a translation. If the cut set is finite, the canonical model is a finite IET; if it
is Lebesgue-null, the canonical model is a countable IET.

The same obstruction has a geometric description. A gap of the support is called active when
its image under the corresponding branch translation has positive measure. Under a
branchwise nonsingularity assumption, the defect support is exactly the closure of the active-gap
cuts. In particular, finite defect support is equivalent to branchwise nonsingularity together
with finiteness of the active-gap cut set. These are sufficient criteria for the canonical
finite/countable realizations; no converse is claimed in the null, non-finite case.

The principal conclusions may be summarized as follows. The notation
\(C_{\mathrm{def}}\) refers to the canonical defect cut set introduced in
Section~7.

\begin{theorem}[Main results]
\label{thm:introduction-main-results}
Let \(S:I\to I\) be an interval translation map and let \(\mu\) be a
non-atomic \(S\)-invariant Borel probability measure. Put
\(J=\supp\mu\), let
\[
        h(x)=\mu([0,x])
\]
be the distribution coordinate, and let \(T=h\circ S\circ h^{-1}\) denote
the induced map modulo the full-measure sets on which this expression is
defined.
\begin{enumerate}[label=\textup{(\roman*)}]
\item If \(C_{\mathrm{def}}\) is finite, then the canonical
      distribution-coordinate model is a finite interval exchange
      transformation modulo endpoints.
\item If
      \[
              \Leb(C_{\mathrm{def}})=0,
      \]
      then the canonical distribution-coordinate model is a countable
      interval exchange transformation modulo Lebesgue-null sets.
\item For the stationary Bruin--Troubetzkoy map studied in
      Section~8.1, let \(\alpha\in(0,1/2)\) be the root of
      \(x^3-x^2-2x+1=0\), put \(\beta=\alpha^2\) and
      \(q=(1-\alpha)^2\), and define
      \[
              u_r=\alpha(1-q^{r+1}),
              \qquad
              v_r=2\alpha-\beta-(\alpha-\beta)q^{r+1}.
      \]
      Then
      \[
              C_{\mathrm{def}}
              =
              \{0,1,\alpha,2\alpha-\beta\}
              \cup\{u_r:r\geq0\}
              \cup\{v_r:r\geq0\}.
      \]
      This set is countably infinite and Lebesgue-null, and every
      active-gap cut in the two displayed sequences is essential.
      Consequently, the canonical model is a countable IET that cannot
      be represented by finitely many translation intervals in the
      distribution coordinate.
\end{enumerate}
\end{theorem}

Parts~\textup{(i)} and~\textup{(ii)} are proved in
Theorem~\ref{thm:defect-set-structure}. Part~\textup{(iii)} follows from
Theorem~\ref{thm:BT-explicit-Cdef} and
Proposition~\ref{prop:BT-nonfinite-canonical}.

The finite-type and infinite-type regimes of interval translation maps have recently received
considerable attention. For three branches, Volk proved that finite-type maps form an open,
dense and full-measure subset of parameter space \cite{Volk2014}. In a recent preprint,
Drach, Staresinic and van Strien prove that, for every number of branches, finite-type maps
contain an open dense subset \cite{DrachStaresinicVanStrien2026}. Infinite-type families with
Cantor attractors exhibit genuinely different ergodic behavior, including weak mixing and
unique ergodicity phenomena \cite{BruinRadinger2026,SkripchenkoChernyi2026}. In particular,
Artigiani, Avila, Ferenczi, Hubert and Skripchenko prove typical weak mixing for the
Bruin--Troubetzkoy family and obtain related spectral results
\cite{ArtigianiAvilaFerencziHubertSkripchenko2026}. These classes provide natural test cases
for the canonical defect cut set introduced here. We carry out such a calculation for a
stationary self-similar Bruin--Troubetzkoy parameter: its canonical defect cut set consists of
two explicit geometric sequences together with their accumulation points. This gives a
canonical countable IET model which is provably non-finite.

The paper is organized as follows. Section~2 fixes the interval conventions and the notions
of finite and countable IET models. Section~3 proves the polynomial coding
bound and the almost-everywhere invertibility theorem, and then derives the abstract countable
model. Section~4 constructs the distribution-coordinate model. Section~5 introduces branch
displacements and a basic finite-piece criterion. Section~6 develops the active-gap criterion.
Section~7 introduces defect measures and proves the main finite and countable structure
theorems. Section~8 records basic examples, computes the canonical defect cut set for a
self-similar infinite-type Bruin--Troubetzkoy map, gives useful geometric reformulations,
and explains why almost-everywhere invertibility does not by itself justify the ordinary
direct-image argument. Section~9 collects open questions. The appendices contain,
respectively, the technical entropy and automorphic-core proofs, the graph-directed
Hausdorff-measure and Bratteli--Vershik details, and the complete renormalization-tower
bookkeeping.

\section{Interval translation maps and metric equivalence}

Throughout the paper, we work on the half-open interval
\[
I=[0,1).
\]
This convention avoids the extra bookkeeping caused by wrap-around in the
circle model. More precisely, each branch translation will be represented by
a real translation whose image remains inside \([0,1)\). The refinement lemma
below shows that this causes no loss of generality for circle interval
translation maps.

\begin{definition}[Interval translation map]
	Let
	\[
	0=t_0<t_1<\cdots<t_n=1,
	\]
	and define
	\[
	I_k=[t_{k-1},t_k),
	\qquad k=1,\ldots,n.
	\]
	A map \(S:I\to I\) is called an \emph{interval translation map} if, for each
	\(k\), there is a constant \(c_k\in\mathbb{R}\) such that
	\[
	S(x)=x+c_k,
	\qquad x\in I_k,
	\]
	with
	\[
	x+c_k\in[0,1)
	\qquad\text{for every }x\in I_k.
	\]
	The finite partition
	\[
	\cP=\{I_1,\ldots,I_n\}
	\]
	is called the \emph{continuity partition}, and
	\[
	\cH=\{t_0,t_1,\ldots,t_n\}
	\]
	is called the \emph{endpoint set}.
\end{definition}

\begin{remark}
	Each interval \(I_k=[t_{k-1},t_k)\) in the continuity partition will be
	called a \emph{branch}, or a \emph{continuity branch}, of \(S\). The points
	\(t_0,\ldots,t_n\) will be called the branch endpoints or branch cuts. On
	each branch, the map is given by a single translation
	\[
	x\longmapsto x+c_k.
	\]
\end{remark}

\begin{remark}
	Although the dynamical system is defined on the half-open interval
	\(I=[0,1)\), it will sometimes be convenient to regard an invariant measure
	as a Borel probability measure on the compact interval \([0,1]\), assigning
	zero mass to the endpoint \(1\). Since all measures considered in this paper
	are non-atomic, this does not change the measured system modulo null sets.
	
	This convention allows us to define the distribution function on the whole
	compact interval by
	\[
	h(x)=\mu([0,x]),
	\qquad x\in[0,1],
	\]
	with \(h(1)=1\).
\end{remark}

The interval convention above also applies to maps originally defined on the
circle. If the image of a branch crosses the point \(0\), we simply divide
that branch at the corresponding wrap-around point.

\begin{lemma}[Finite refinement of circle branches]
	\label{lem:circle-refinement}
	Suppose that a circle interval translation map is defined on finitely many
	branches by
	\[
	S(x)=x+c_k\pmod 1.
	\]
	By cutting each branch, when necessary, at the preimage of the point where
	its image crosses \(0\), one obtains a finite refinement of the original
	partition with the following property: on every refined branch, \(S\) is
	represented by a real translation
	\[
	x\longmapsto x+d
	\]
	whose image lies in \([0,1)\).
	
	Only finitely many new endpoints are introduced. Consequently, for every
	non-atomic invariant measure, the original and refined systems agree modulo
	null sets.
\end{lemma}

\begin{proof}
	Fix one of the original branches \(I_k\). Consider the lifted translation
	\[
	x\longmapsto x+c_k.
	\]
	Either its image lies entirely inside one translate of \([0,1)\), or it
	crosses an integer boundary. Since the image of \(I_k\) is an interval of
	length at most \(1\), there is at most one point of \(I_k\) at which
	\(x+c_k\) is an integer.
	
	In the second case, we cut \(I_k\) at this point. The original branch is then
	divided into at most two subintervals. On each subinterval, the circle map is
	represented by an ordinary real translation, namely either
	\[
	x\longmapsto x+c_k
	\]
	or
	\[
	x\longmapsto x+c_k-1,
	\]
	and the image of that translation lies in \([0,1)\).
	
	Because the original partition has only finitely many branches, this
	procedure introduces only finitely many new endpoints. A non-atomic measure
	assigns zero mass to every finite set, so the refinement does not change the
	measured dynamics modulo null sets.
\end{proof}

\begin{remark}[Dynamical and compact attractors]
\label{rem:dynamical-compact-attractors}
For an interval translation map
\[
        S:I=[0,1)\longrightarrow I,
\]
we distinguish between the usual dynamical attractor
\[
        \Omega(S)
        :=
        \bigcap_{m\geq0}S^m(I)
\]
and its compact topological counterpart
\[
        J_{\mathrm{att}}(S)
        :=
        \bigcap_{m\geq0}\overline{S^m(I)}^{\,[0,1]}.
\]
Here the closures are taken in the compact interval \([0,1]\).

The set \(\Omega(S)\) is the intersection of the actual forward images
of the half-open dynamical interval \(I\). It need not be closed in
\([0,1]\). By contrast, \(J_{\mathrm{att}}(S)\) is compact, since it is
an intersection of nested nonempty compact subsets of \([0,1]\). We
always have
\[
        \Omega(S)\subset J_{\mathrm{att}}(S),
\]
but no equality between these two sets is assumed unless it is proved
in the particular situation under consideration.

In statements concerning the dynamics on \(I\), we use
\(\Omega(S)\). When endpoints, compact graph-directed constructions,
Hausdorff measure, or complementary gaps are involved, we use
\(J_{\mathrm{att}}(S)\). In Section~8.1 we retain the notation
\(J_{\mathrm{att}}\) for the compact attractor until it is identified with
the support of the invariant measure constructed there. Only after that identification do we
write
\[
        J=\supp\mu=J_{\mathrm{att}}(S),
\]
with
\[
        J_{\mathrm{dyn}}:=J_{\mathrm{att}}(S)\cap I
\]
for the actual dynamical part.
\end{remark}

\begin{definition}[Interval exchange transformation (IET)]
	An interval translation map \(T\) is called an \emph{interval exchange transformation (IET)}
	if it is one-to-one except possibly at branch endpoints. Equivalently, the
	images of its continuity intervals are pairwise disjoint modulo
	Lebesgue-null sets and cover the interval modulo endpoints.
\end{definition}

\begin{definition}[Countable interval exchange transformation (countable IET)]\label{def:countable-iem}
A measurable map $R:[0,1]\to[0,1]$ is a countable interval exchange transformation (countable IET)
modulo Lebesgue-null sets if there is a countable family of pairwise disjoint open
intervals $\{U_m\}_{m\geq1}$ whose union has full Lebesgue measure, and constants $d_m$,
such that $R(u)=u+d_m$ for almost every $u\in U_m$, while the translated intervals
$U_m+d_m$ are contained in $[0,1]$ up to endpoints, are pairwise disjoint modulo
endpoints, and cover $[0,1]$ modulo a null set.
\end{definition}

\begin{definition}[Metric equivalence]
	Let
	\[
	(X,\borel_X,\mu,S)
	\qquad\text{and}\qquad
	(Y,\borel_Y,\nu,T)
	\]
	be measure-preserving systems on standard Borel probability spaces. We say
	that they are \emph{metrically equivalent} if there exist invariant
	full-measure sets
	\[
	X_0\subset X,
	\qquad
	Y_0\subset Y,
	\]
	and a measurable bijection
	\[
	\Phi:X_0\longrightarrow Y_0
	\]
	with measurable inverse such that
	\[
	\Phi_*\mu=\nu
	\]
	and
	\[
	\Phi\circ S=T\circ\Phi
	\]
	on \(X_0\).
\end{definition}

Throughout the rest of the paper, \(\mu\) denotes a non-atomic
\(S\)-invariant Borel probability measure, and
\[
J=\operatorname{supp}\mu
\]
denotes its support, viewed as a closed subset of the compact interval
\([0,1]\). Consequently,
\[
[0,1]\setminus J
\]
is relatively open in \([0,1]\) and has at most countably many connected
components. Each component is a relatively open interval in \([0,1]\); it
may have the form
\[
(a,b),\qquad [0,b),\qquad\text{or}\qquad(a,1].
\]
We call these connected components the \emph{gaps} of \(J\).
\section{Almost-everywhere invertibility}\label{sec:ae-invertibility}

Almost-everywhere invertibility of interval translation maps with respect to non-atomic
invariant probability measures was previously proved by Tajbakhsh and Yaghmaeian
\cite{TajbakhshYaghmaeian2025}. We include here a different, self-contained proof based on
the complexity of the natural branch coding. Besides providing an independent argument, this
approach identifies explicitly the measurable inverse and the invariant full-measure
automorphic core that will be used in Subsection~3.1 and later in the analysis of direct branch
images.

The proof proceeds in three steps. First, we show that the number of length-$m$ branch
itineraries grows at most polynomially, and hence the entropy of the continuity partition is
zero. A conditional-entropy argument then shows that the present branch of a point is
determined almost everywhere by its future itinerary, equivalently by the image point together
with its future. This allows us to select the correct inverse translation branch measurably.
Finally, after modifying the maps only on null sets, we obtain an invariant full-measure subset
on which the original interval translation map is a measure-preserving automorphism.

This argument is consistent with Buzzi's zero-entropy theorem for piecewise isometries
\cite{Buzzi2001}, but it does not require an identification of different notions of topological
entropy. Let $\cP=\{I_1,\ldots,I_n\}$ be the continuity partition. For $m\geq1$, define
\[
        \cP^{(m)}=\bigvee_{j=0}^{m-1}S^{-j}\cP.
\]
The atoms of $\cP^{(m)}$ are the sets of points whose first $m$ branch symbols are prescribed.

\begin{lemma}[Polynomial growth of the natural coding]\label{lem:polynomial-coding}
Let \(S\) be an interval translation map with continuity partition
\(\cP=\{I_1,\ldots,I_n\}\), branch translations \(x\mapsto x+c_k\), and endpoint set
\[
        \cH=\{t_0,\ldots,t_n\}.
\]
Then
\[
        \#\cP^{(m)}
        \leq
        1+2(n+1)\binom{m+n-1}{n}.
\]
In particular,
\[
        \lim_{m\to\infty}\frac1m\log \#\cP^{(m)}=0.
\]
\end{lemma}

\begin{proof}
For a finite partition \(\mathcal Q\), let
\[
        \partial_I\mathcal Q
        =\bigcup_{Q\in\mathcal Q}\partial_I Q,
\]
where boundaries are taken relative to \(I=[0,1)\). We first prove the one-step inclusion
\[
        \partial_I(S^{-1}A)
        \subset
        \cH\cup S^{-1}(\partial_I A)
\tag{3.1}
\label{eq:one-step-boundary}
\]
for every set \(A\subset I\). Indeed, if
\(x\notin\cH\cup S^{-1}(\partial_I A)\), then \(x\) lies in the interior of one branch,
and on a neighborhood of \(x\) the map \(S\) is a translation. Since
\(S(x)\notin\partial_I A\), after shrinking the neighborhood its image lies either entirely
inside \(A\) or entirely outside \(A\). Hence \(x\notin\partial_I(S^{-1}A)\), proving
\eqref{eq:one-step-boundary}.

We now claim that, for every \(j\geq0\),
\[
        \partial_I(S^{-j}\cP)
        \subset
        \bigcup_{r=0}^{j}S^{-r}(\cH).
\tag{3.2}
\label{eq:pullback-boundary-induction}
\]
For \(j=0\), this is just \(\partial_I\cP\subset\cH\). If it holds for \(j\), then
\eqref{eq:one-step-boundary}, applied to every member of \(S^{-j}\cP\), gives
\[
\begin{aligned}
\partial_I(S^{-(j+1)}\cP)
&\subset
\cH\cup S^{-1}\bigl(\partial_I(S^{-j}\cP)\bigr)\\
&\subset
\cH\cup S^{-1}\left(\bigcup_{r=0}^{j}S^{-r}(\cH)\right)\\
&=
\bigcup_{r=0}^{j+1}S^{-r}(\cH).
\end{aligned}
\]
This completes the induction.

An atom of the join
\[
        \cP^{(m)}=\bigvee_{j=0}^{m-1}S^{-j}\cP
\]
is an intersection of one member from each pulled-back partition. Since the boundary of a
finite intersection is contained in the union of the individual boundaries,
\[
        \partial_I\cP^{(m)}
        \subset
        D_m:=\bigcup_{r=0}^{m-1}S^{-r}(\cH).
\tag{3.3}
\label{eq:refinement-boundary}
\]
Consequently, on every connected component of \(I\setminus D_m\), the first \(m\) branch
symbols are constant.

It remains to estimate \(D_m\). Fix \(j\geq0\) and let \(x\in S^{-j}(\cH)\). If the first
\(j\) branch symbols of \(x\) are \(k_0,\ldots,k_{j-1}\), then along this itinerary
\[
        S^j(x)=x+\sum_{r=0}^{j-1}c_{k_r}.
\]
Thus, for some endpoint \(t_\ell\in\cH\),
\[
        x=t_\ell-\sum_{r=0}^{j-1}c_{k_r}
         =t_\ell-\sum_{k=1}^n a_kc_k,
\]
where \(a_k\) is the number of visits to branch \(I_k\) during the first \(j\) iterates.
Hence \(a_k\geq0\) and
\[
        a_1+\cdots+a_n=j.
\]
For fixed \(j\), the number of possible vectors \((a_1,\ldots,a_n)\) is
\(\binom{j+n-1}{n-1}\). Different itineraries may produce the same vector or the same point,
so this gives an upper bound. Therefore
\[
        \#D_m
        \leq
        (n+1)\sum_{j=0}^{m-1}\binom{j+n-1}{n-1}
        =
        (n+1)\binom{m+n-1}{n}.
\]
The components of \(I\setminus D_m\) contribute at most \(\#D_m+1\) atoms of
\(\cP^{(m)}\), while points of \(D_m\cap I\) contribute at most \(\#D_m\) further atoms.
Hence
\[
        \#\cP^{(m)}
        \leq
        2\#D_m+1
        \leq
        1+2(n+1)\binom{m+n-1}{n}.
\]
The right-hand side grows polynomially in \(m\), proving the asserted limit.
\end{proof}

\begin{remark}
The preceding estimate is a direct one-dimensional coding argument. It is consistent
with Buzzi's general theorem that piecewise isometries have zero topological entropy
\cite{Buzzi2001}, but no identification between different notions of topological entropy
is needed in the proof below.
\end{remark}

\begin{lemma}[Zero entropy of the continuity partition]\label{lem:zero-partition-entropy}
Let $\mu$ be an $S$-invariant Borel probability measure. Then
\[
        h_\mu(S,\cP)=0.
\]
\end{lemma}

\begin{proof}
For every $m\geq1$, Shannon entropy is bounded by the logarithm of the number of atoms:
\[
        H_\mu(\cP^{(m)})\leq \log \#\cP^{(m)}.
\]
Therefore
\[
        0\leq \frac1m H_\mu(\cP^{(m)})\leq \frac1m\log \#\cP^{(m)}.
\]
By Lemma~\ref{lem:polynomial-coding}, the right-hand side tends to $0$. Since
\[
        h_\mu(S,\cP)=\lim_{m\to\infty}\frac1m H_\mu(\cP^{(m)}),
\]
the result follows.
\end{proof}

\begin{lemma}[Entropy identity]\label{lem:entropy-identity}
Let $S$ be a measure-preserving transformation and let $\cP$ be a finite measurable partition. Then
\[
        h_\mu(S,\cP)=H_\mu\left(\cP\,\middle|\,\bigvee_{j=1}^{\infty}S^{-j}\cP\right).
\]
\end{lemma}

\begin{proof}
The complete proof is given in Appendix~\ref{app:entropy-automorphic-details}.
\end{proof}

\begin{lemma}[Zero conditional entropy determines the branch]\label{lem:branch-future}
Let $\alpha:I\to\{1,\ldots,n\}$ be the branch label,
\[
        \alpha(x)=k\quad\Longleftrightarrow\quad x\in I_k.
\]
If $h_\mu(S,\cP)=0$, then $\alpha$ is measurable modulo $\mu$ with respect to the future sigma-algebra
\[
        \cP^+=\bigvee_{j=1}^{\infty}S^{-j}\cP.
\]
Equivalently, the present branch of $x$ is determined $\mu$-almost everywhere by the future itinerary $Sx,S^2x,S^3x,\ldots$.
\end{lemma}

\begin{proof}
The complete proof is given in Appendix~\ref{app:entropy-automorphic-details}.
\end{proof}

\begin{lemma}[The branch is a function of the image point]\label{lem:branch-image}
There exists a measurable function
\[
        \beta:I\to\{1,\ldots,n\}
\]
such that
\[
        \alpha(x)=\beta(Sx)
\]
for $\mu$-almost every $x$.
\end{lemma}

\begin{proof}
The complete proof is given in Appendix~\ref{app:entropy-automorphic-details}.
\end{proof}

\begin{lemma}[Automorphic core from two one-sided inverses]\label{lem:automorphic-core}
Let \((Z,\mathcal A,\mu)\) be a standard Borel probability space and let
\(S,R:Z\to Z\) be measurable maps. Assume that \(S\) preserves \(\mu\), and that
\[
        R\circ S=\mathrm{id},\qquad S\circ R=\mathrm{id}
\]
hold \(\mu\)-almost everywhere. Then there exists a full-measure Borel set
\(Z_0\subset Z\) such that \(S(Z_0)=Z_0\), \(R(Z_0)=Z_0\), and the restrictions
\(S|_{Z_0}\) and \(R|_{Z_0}\) are inverse measurable measure-preserving bijections.
\end{lemma}

\begin{proof}
The complete proof is given in Appendix~\ref{app:entropy-automorphic-details}.
\end{proof}

\begin{theorem}[Almost-everywhere invertibility]\label{thm:ae-invertible}
Let $S$ be an interval translation map and let $\mu$ be a non-atomic $S$-invariant Borel probability measure. Put $J=\supp\mu$. Then $S$ is invertible $\mu$-almost everywhere on $J$. More precisely, there exists an invariant full-measure Borel set $X_0\subset J$ such that
\[
        S:X_0\to X_0
\]
is a measurable bijection whose inverse is measurable and measure-preserving.
\end{theorem}

\begin{proof}
By Lemma~\ref{lem:zero-partition-entropy}, \(h_\mu(S,\cP)=0\). Hence
Lemma~\ref{lem:branch-image} gives a measurable function
\(\beta:I\to\{1,\ldots,n\}\) such that
\[
        \alpha(x)=\beta(Sx)
\]
for \(\mu\)-almost every \(x\).

Choose a fixed point \(x_*\in I\). For every branch \(I_k\), define a measurable map
\(R_k:I\to I\) by
\[
R_k(y)=
\begin{cases}
        y-c_k, & y\in S(I_k),\\
        x_*,   & y\notin S(I_k).
\end{cases}
\]
On \(S(I_k)\), this is the inverse of the restriction \(S|_{I_k}\). Define
\[
        R(y)=R_{\beta(y)}(y).
\]
Since \(\beta\) has finite range and each \(R_k\) is measurable, \(R:I\to I\) is
measurable.

Let \(E\subset J\) be a Borel full-measure set on which
\(\alpha(x)=\beta(Sx)\). If \(x\in E\cap I_k\), then \(\beta(Sx)=k\) and
\(Sx\in S(I_k)\), so
\[
        R(Sx)=R_k(Sx)=x.
\]
Thus \(R\circ S=\mathrm{id}\) \(\mu\)-almost everywhere on \(J\).

Next set
\[
        F=\{y\in I:S(Ry)=y\}.
\]
The set \(F\) is Borel. If \(x\in E\), then
\(S(R(Sx))=Sx\), so \(Sx\in F\), equivalently \(E\subset S^{-1}F\). By
invariance,
\[
        \mu(F)=\mu(S^{-1}F)\geq\mu(E)=1.
\]
Hence \(S\circ R=\mathrm{id}\) \(\mu\)-almost everywhere.

We now replace \(S\) and \(R\), on null sets only, by measurable
self-maps of \(J\). Put
\[
        N_S
        :=
        \{x\in J:Sx\notin J\}.
\]
Since \(\mu(J)=1\) and \(S\) preserves \(\mu\),
\[
        \mu(N_S)
        =
        \mu\bigl(J\setminus S^{-1}J\bigr)
        =
        0.
\]

Moreover,
\[
        S(E)
        =
        \bigcup_{k=1}^{n}\tau_k(E\cap I_k)
\]
is Borel. Since
\[
        E\subset S^{-1}(S(E)),
\]
invariance gives
\[
        \mu(S(E))
        =
        \mu\bigl(S^{-1}(S(E))\bigr)
        \geq
        \mu(E)
        =
        1.
\]
Thus \(S(E)\cap J\) has full measure in \(J\). For every
\(y=Sx\) with \(x\in E\), we have
\[
        R(y)=R(Sx)=x\in J.
\]
Consequently, the Borel set
\[
        N_R
        :=
        \{y\in J:R(y)\notin J\}
\]
is \(\mu\)-null.

Choose a fixed point \(z_*\in J\) and define
\[
        \widetilde S(x)
        :=
        \begin{cases}
        S(x),&x\in J\setminus N_S,\\
        z_*,&x\in N_S,
        \end{cases}
\]
and
\[
        \widetilde R(y)
        :=
        \begin{cases}
        R(y),&y\in J\setminus N_R,\\
        z_*,&y\in N_R.
        \end{cases}
\]
Then \(\widetilde S,\widetilde R:J\to J\) are measurable.

We verify explicitly that the two almost-everywhere inverse identities
are unchanged. Let
\[
        A
        :=
        \{x\in J:R(Sx)=x\}.
\]
The set \(A\) has full measure. If
\[
        x\notin
        (J\setminus A)\cup N_S\cup S^{-1}(N_R),
\]
then
\[
        \widetilde S(x)=S(x),
        \qquad
        \widetilde R(Sx)=R(Sx),
\]
and hence
\[
        \widetilde R(\widetilde Sx)=x.
\]
Now
\[
        \mu(S^{-1}(N_R))=\mu(N_R)=0
\]
by the \(S\)-invariance of \(\mu\). Therefore
\[
        \widetilde R\circ\widetilde S
        =
        \operatorname{id}
        \qquad
        \mu\text{-almost everywhere on }J.
\]

Similarly, let
\[
        B
        :=
        \{y\in J:S(Ry)=y\}.
\]
The set \(B\) has full measure. If
\[
        y\in B\setminus N_R,
\]
then \(R(y)\in J\) and
\[
        S(Ry)=y\in J.
\]
It follows that \(R(y)\notin N_S\). Hence neither modification affects
the composition at \(y\), and
\[
        \widetilde S(\widetilde Ry)
        =
        S(Ry)
        =
        y.
\]
Thus
\[
        \widetilde S\circ\widetilde R
        =
        \operatorname{id}
        \qquad
        \mu\text{-almost everywhere on }J.
\]

Finally, \(\widetilde S\) still preserves \(\mu\). Indeed, for every
Borel set \(C\subset J\),
\[
        \widetilde S^{-1}(C)
        \mathbin{\triangle}
        \bigl(S^{-1}(C)\cap J\bigr)
        \subset N_S,
\]
and therefore
\[
        \mu(\widetilde S^{-1}(C))
        =
        \mu(S^{-1}(C))
        =
        \mu(C).
\]
Renaming \(\widetilde S\) and \(\widetilde R\) as \(S\) and \(R\), we
may therefore apply Lemma~\ref{lem:automorphic-core} on the standard
Borel probability space \((J,\mu)\).

Lemma~\ref{lem:automorphic-core} gives an invariant full-measure Borel
set \(X_0\subset J\) on which \(S\) and \(R\) are inverse measurable
bijections. Since \(S\) preserves \(\mu\), the bijection
\[
        S:X_0\longrightarrow X_0
\]
is measure preserving.
\end{proof}

\begin{corollary}[Direct-image invariance on the invertible core]\label{cor:direct-core}
Let $X_0$ be as in Theorem~\ref{thm:ae-invertible}. If $A\subset X_0$ is Borel in the relative Borel structure, then
\[
        \mu(SA)=\mu(A).
\]
\end{corollary}

\begin{proof}
The restriction \(S|_{X_0}\) is a Borel bijection with Borel inverse, so \(S(A)\) is
Borel in \(X_0\). Since the restriction is measure preserving and
\((S|_{X_0})^{-1}(S(A))=A\), we obtain \(\mu(SA)=\mu(A)\).
\end{proof}

\subsection{A universal countable IET model}

Theorem~\ref{thm:ae-invertible} supplies the only point at which the interval-translation
structure is needed in this subsection: after discarding a null set, the measured ITM becomes
a non-atomic measure-preserving automorphism of a standard Borel probability space. From that
point onward, the representation problem is purely measure-theoretic.

The principal external input is the theorem of Arnoux, Ornstein and Weiss
\cite{ArnouxOrnsteinWeiss1985}. It asserts that every aperiodic non-atomic
measure-preserving automorphism of a standard probability space admits a measurable model as
an infinite interval exchange transformation. In the terminology of
Definition~\ref{def:countable-iem}, this is a countable IET model for the aperiodic component.
A general automorphism may also have periodic components, so the AOW theorem cannot simply be
applied to the whole space without a decomposition. We separate the space into its aperiodic
part and the sets of points of exact period \(n\). The aperiodic part is represented by AOW,
while each period-\(n\) component is represented by a finite cyclic exchange of \(n\) interval
levels. Placing all non-null component models consecutively in \([0,1)\) yields a single
countable IET.

Thus almost-everywhere invertibility provides the automorphic core, and the AOW theorem,
together with the elementary periodic decomposition, provides the existence of an abstract
countable-IET representation. This model is not claimed to preserve the order inherited from
the original interval, and its conjugating coordinate need not be the distribution function
introduced in Section~\ref{sec:distribution-model}. The canonical order-preserving problem is
therefore treated separately in the subsequent sections.

\begin{theorem}[Countable IET model for automorphisms]
\label{thm:automorphism-countable-iem}
Let \((Z,\mathcal A,\nu,U)\) be a non-atomic measure-preserving automorphism of a standard
Borel probability space. Then it is metrically equivalent to a countable IET in the sense of Definition~\ref{def:countable-iem}.
\end{theorem}

\begin{proof}
For \(n\geq1\), let
\[
P_n=\{z\in Z:U^n z=z\text{ and }U^jz\neq z\text{ for }1\leq j<n\}
\]
be the set of points of exact period \(n\), and put
\[
        A=Z\setminus\bigcup_{n\geq1}P_n.
\]
The sets \(A,P_1,P_2,\ldots\) are Borel, pairwise disjoint and invariant, and \(U|_A\)
is aperiodic.

Suppose first that \(a:=\nu(A)>0\), and normalize \(\nu|_A\) to the probability measure
\(\nu_A=\nu|_A/a\). By the theorem of Arnoux, Ornstein and Weiss
\cite{ArnouxOrnsteinWeiss1985}, the aperiodic system \((A,\nu_A,U|_A)\) is
measure-theoretically isomorphic to an infinite IET \(R_A\) on \([0,1)\). Conjugating by
the affine map \(L_a(u)=au\) gives a countable IET on \([0,a)\), and the resulting
coordinate sends \(\nu|_A\) to Lebesgue measure on that interval. If \(a=0\), this block is
omitted.

Now fix \(n\geq1\) with \(p_n:=\nu(P_n)>0\). Because \(Z\) is standard Borel, choose a
Borel linear order on \(Z\), and let
\[
F_n=\{z\in P_n:z<U^jz\text{ for every }1\leq j<n\}.
\]
Then \(F_n\) contains exactly one point from each period-\(n\) orbit and
\[
        P_n=F_n\sqcup UF_n\sqcup\cdots\sqcup U^{n-1}F_n.
\]
Since \(U\) preserves \(\nu\),
\[
        \nu(F_n)=\frac{p_n}{n}=: \ell_n.
\]
The standard non-atomic finite measure space \((F_n,\nu|_{F_n})\) is measure-isomorphic
modulo null sets to \([0,\ell_n)\) by the standard isomorphism theorem for non-atomic
standard measure spaces \cite{CornfeldFominSinai1982}; let \(\theta_n\) be such an
isomorphism. Divide \(I_n=[0,p_n)\) into
\[
        I_{n,j}=[j\ell_n,(j+1)\ell_n),\qquad 0\leq j<n,
\]
and define
\[
        \Psi_n(U^jz)=\theta_n(z)+j\ell_n.
\]
This conjugates \(U|_{P_n}\) to the cyclic IET
\[
R_n(u)=
\begin{cases}
        u+\ell_n,&u\in I_{n,j},\quad 0\leq j<n-1,\\
        u-(n-1)\ell_n,&u\in I_{n,n-1}.
\end{cases}
\]

Finally, place the non-null intervals corresponding to \(A,P_1,P_2,\ldots\) consecutively
inside \([0,1)\), preserving their lengths, and translate the component models to these
blocks. On each non-null component, choose a full-measure set \(E\) on which the
componentwise coordinate is a measurable bijection and the conjugacy identity holds, and
replace it by
\[
        E_\infty=\bigcap_{m\in\mathbb Z}U^{-m}E.
\]
Because \(U\) is an automorphism, \(E_\infty\) is invariant and has full measure. Its image
under the componentwise coordinate is likewise an invariant full-measure set for the model
IET, and the restricted coordinate is a pointwise conjugacy. Performing this construction on
all non-null components and deleting the resulting countable union of null exceptional sets,
the blockwise coordinate is a metric conjugacy. The aperiodic block contributes countably
many exchange intervals, and each periodic block contributes finitely many; hence the total
family is countable.
\end{proof}

\begin{corollary}[Universal countable IET model]
\label{cor:universal-countable-iem}
Let \(S\) be an interval translation map and let \(\mu\) be a non-atomic
\(S\)-invariant Borel probability measure. Put \(J=\supp\mu\). Then
\((J,\mu,S)\) is metrically equivalent to a countable IET.
\end{corollary}

\begin{proof}
By Theorem~\ref{thm:ae-invertible}, the system has a full-measure invariant standard Borel
subset on which \(S\) is a non-atomic measure-preserving automorphism. Apply
Theorem~\ref{thm:automorphism-countable-iem}.
\end{proof}

\section{The distribution-coordinate model}\label{sec:distribution-model}

In this section we construct the map that transfers the invariant measure of the
interval translation map to Lebesgue measure. This map will be used to rewrite the
measured system on the support of the invariant measure as a Lebesgue-preserving
interval system. The construction is canonical: it is given by the distribution
function of the invariant measure.

\begin{definition}[Distribution coordinate]
Define
\[
        h:[0,1]\to[0,1],\qquad h(x)=\mu([0,x]).
\]
\end{definition}

\begin{lemma}[Distribution coordinate is a metric isomorphism]\label{lem:h-isomorphism}
Let $\mu$ be a non-atomic Borel probability measure on $[0,1]$ and let $J=\supp\mu$. Then $h$ is continuous, non-decreasing, onto, and $h_*\mu=\Leb$. Moreover, there exist full-measure Borel sets $J_*\subset J$ and $Y_*\subset[0,1]$ such that
\[
        h:J_*\to Y_*
\]
is a measurable bijection with measurable inverse. In particular, $h$ is a measure-space isomorphism from $(J,\mu)$ to $([0,1],\Leb)$ modulo null sets.
\end{lemma}

\begin{proof}
Monotonicity is immediate. Since $\mu$ is non-atomic, $h$ has no jumps. More explicitly, if $x_m\downarrow x$, then $[0,x_m]\downarrow[0,x]$, so $h(x_m)\to h(x)$. If $x_m\uparrow x$, then $[0,x_m]\uparrow[0,x)$, so
\[
        h(x_m)\to\mu([0,x))=\mu([0,x])=h(x).
\]
Thus $h$ is continuous. We view $\mu$ as a probability measure on the compact interval $[0,1]$; since $\mu$ is non-atomic, $h(0)=\mu(\{0\})=0$ and $h(1)=1$. Continuity therefore gives that $h$ is onto.

To prove $h_*\mu=\Leb$, let $0\leq t\leq1$ and set
\[
        q(t)=\sup\{x\in[0,1]:h(x)\leq t\}.
\]
By continuity of $h$, $h(q(t))=t$. The set $h^{-1}([0,t])$ is an initial interval with endpoint $q(t)$, up to endpoints, and endpoints have $\mu$-measure zero. Hence
\[
        \mu(h^{-1}([0,t]))=\mu([0,q(t)])=h(q(t))=t=\Leb([0,t]).
\]
The equality extends from intervals $[0,t]$ to all Borel sets by the monotone class theorem. Thus $h_*\mu=\Leb$.

It remains to prove injectivity modulo null sets on the support. Since \(J\) is compact in \([0,1]\), its complement
\[
        [0,1]\setminus J
\]
is relatively open in \([0,1]\). Therefore it is the union of its
pairwise disjoint connected components, each of which is a relatively
open interval in \([0,1]\). We call these components the gaps of \(J\).
Thus every gap has one of the forms
\[
        [0,a),\qquad
        (a,b),\qquad
        (b,1],
\]
with the obvious omission of the boundary cases when \(0,1\in J\).
Let \(E\) be the countable set of all gap endpoints and put
\[
        J_*=J\setminus E.
\]
Since $\mu$ is non-atomic, $\mu(E)=0$, so $J_*$ has full $\mu$-measure in $J$.

We claim that $h$ is injective on $J_*$. Suppose $x,y\in J_*$, $x<y$, and $h(x)=h(y)$. Then
\[
        \mu((x,y))=0.
\]
If $(x,y)$ met $J$, then some point $z\in (x,y)\cap J$ would have every neighborhood of positive $\mu$-measure. Choosing such a neighborhood inside $(x,y)$ would contradict $\mu((x,y))=0$. Hence $(x,y)\cap J=\emptyset$, so $(x,y)$ is contained in a gap of $J$. Since $x,y\in J$, they must be the two endpoints of that gap. This contradicts $x,y\notin E$. Thus $h$ is injective on $J_*$.

The image $Y_*=h(J_*)$ is Borel by the Lusin--Souslin theorem \cite{Kechris1995}, because $h$ is Borel and injective on the Borel set $J_*$. The inverse $h^{-1}:Y_*\to J_*$ is also Borel. Since $h_*\mu=\Leb$, the set $Y_*$ has full Lebesgue measure. This proves the measure-space isomorphism statement.
\end{proof}

\begin{definition}[The induced interval map]\label{def:induced-T}
Let $X_0\subset J$ be the invariant full-measure set from Theorem~\ref{thm:ae-invertible}, and let $E_0$ be the countable set consisting of the branch endpoints and the gap endpoints of $J$. Remove from $X_0$ the full orbit of $E_0\cap X_0$ under the automorphism $S:X_0\to X_0$, namely
\[
        \bigcup_{m\in\mathbb Z}S^m(E_0\cap X_0).
\]
The resulting set will be denoted by $X$. Then $X$ is invariant, has full $\mu$-measure, and $h$ is injective on $X$.

Put
\[
        Y=h(X).
\]
Then $Y$ has full Lebesgue measure. Let
\[
        q:Y\to X
\]
be the measurable inverse of $h:X\to Y$. Define
\[
        T:Y\to Y,
        \qquad
        T(u)=h(S(q(u))).
\]
Fix \(u_*\in[0,1]\) and set \(T(u)=u_*\) for \(u\notin Y\). Since \(Y\) is Borel,
this gives a Lebesgue-measurable extension of \(T\) to \([0,1]\).
\end{definition}

\begin{theorem}[The distribution coordinate gives a metric model]\label{thm:metric-model}
The systems $(J,\mu,S)$ and $([0,1],\Leb,T)$ are metrically equivalent modulo null sets. In particular, $T$ preserves Lebesgue measure and
\[
        T\circ h=h\circ S
\]
on $X$.
\end{theorem}

\begin{proof}
By construction, $h:X\to Y$ is a measurable bijection with measurable inverse $q$, and $h_*\mu=\Leb$. For $x\in X$,
\[
        T(h(x))=h(S(q(h(x))))=h(Sx).
\]
Thus $T\circ h=h\circ S$ on $X$. Since $X$ and $Y$ are invariant full-measure sets and $h$ is a measure-space isomorphism between them, the systems are metrically equivalent.

Finally, the values of $T$ on the null set $[0,1]\setminus Y$ do not affect Lebesgue measure. For every Borel set $B\subset[0,1]$, using the conjugacy identity on the invariant full-measure set $X$ gives
\[
\begin{aligned}
\Leb(T^{-1}B)
&=\Leb(T^{-1}B\cap Y) \\
&=\mu(h^{-1}(T^{-1}B\cap Y)) \\
&=\mu(S^{-1}(h^{-1}(B\cap Y))) \\
&=\mu(h^{-1}(B\cap Y)) \\
&=\Leb(B\cap Y)=\Leb(B).
\end{aligned}
\]
Hence $T$ preserves Lebesgue measure.
\end{proof}

\section{Branch defects and a finite-IET criterion}

Theorem~\ref{thm:metric-model} reduces the problem to understanding when the induced map $T$ is a finite IET. This section gives a general sufficient finite-piece criterion in terms of branch defect functions.

For a branch $I_k$, define
\[
        \tau_k(x)=x+c_k,
        \qquad x\in I_k.
\]
Thus $S=\tau_k$ on $I_k$.

\begin{definition}[Branch defect]
For $x\in I_k$, define
\[
        \Delta_k(x)=h(\tau_k x)-h(x)=h(x+c_k)-h(x).
\]
The function $\Delta_k$ measures the failure of the distribution coordinate to turn the branch translation $x\mapsto x+c_k$ into a literal translation by a constant.
\end{definition}

\begin{lemma}[Local translation criterion]\label{lem:local-translation-criterion}
Let $U\subset[0,1]$ be an open interval such that $q(U\cap Y)$ is contained in a single continuity branch $I_k$ and such that $\Delta_k$ is constant on $q(U\cap Y)$. Then $T$ is a translation on $U$ modulo Lebesgue-null sets.
\end{lemma}

\begin{proof}
Let $u\in U\cap Y$ and put $x=q(u)$. Since $x\in I_k$,
\[
        T(u)=h(Sx)=h(x+c_k)=h(x)+\Delta_k(x)=u+\Delta_k(x).
\]
By hypothesis $\Delta_k$ is constant on $q(U\cap Y)$, say equal to $d$. Hence
\[
        T(u)=u+d
\]
for every $u\in U\cap Y$. Since $Y$ has full Lebesgue measure, this says that $T$ agrees almost everywhere on $U$ with a translation.
\end{proof}

\begin{definition}[Finite branch-defect condition]\label{def:finite-defect}
We say that $S$ satisfies the \emph{finite branch-defect condition} with respect to $\mu$ if there is a finite set
\[
        \cC\subset[0,1]
\]
containing all branch cut values $h(t_0),\ldots,h(t_n)$ such that, for every connected component $U$ of $[0,1]\setminus\cC$, the set $q(U\cap Y)$ is contained in one continuity branch $I_k$ and the defect $\Delta_k$ is constant on $q(U\cap Y)$.
\end{definition}

\begin{theorem}[Finite branch-defect criterion]\label{thm:finite-defect}
If the finite branch-defect condition holds, then $(J,\mu,S)$ is metrically equivalent to a finite IET.
\end{theorem}

\begin{proof}
Let $\cC$ be the finite set from Definition~\ref{def:finite-defect}. The connected
components of $[0,1]\setminus\cC$ form a finite collection of open intervals. By
Lemma~\ref{lem:local-translation-criterion}, on each such component \(U\) the map \(T\)
agrees almost everywhere with a translation \(u\mapsto u+d_U\). Necessarily
\(U+d_U\subseteq[0,1]\) up to endpoints: otherwise a nonempty open subinterval of \(U\)
would be sent outside \([0,1]\), contradicting that \(T\) is \([0,1]\)-valued almost
everywhere on \(U\). Redefining \(T\) on the finite set \(\cC\) and on a Lebesgue-null
set, we obtain a finite interval translation map of \([0,1]\). Since this modification is
confined to a null set, it does not change Lebesgue preservation.

By Theorem~\ref{thm:metric-model}, the resulting representative preserves Lebesgue
measure. A finite interval translation map preserving Lebesgue measure is an
interval exchange transformation, after choosing half-open
representatives and modifying only on endpoints.

Indeed, suppose that the image intervals of two distinct continuity
intervals \(U_i\) and \(U_j\) have an intersection containing a Borel
set \(B\) of positive Lebesgue measure. Since \(T\) is a translation on
each of \(U_i\) and \(U_j\), the preimage \(T^{-1}(B)\) contains two
disjoint sets, one in each branch, each having Lebesgue measure
\(\operatorname{Leb}(B)\). Therefore
\[
        \operatorname{Leb}(T^{-1}(B))
        \geq
        2\operatorname{Leb}(B)
        >
        \operatorname{Leb}(B),
\]
contradicting the Lebesgue invariance of \(T\).

Thus distinct branch images overlap only in Lebesgue-null sets. Since
\(T\) preserves Lebesgue probability measure, the union of the branch
images has full Lebesgue measure. Hence \(T\) is a finite interval
exchange transformation modulo endpoints. The metric equivalence with
\((J,\mu,S)\) follows from Theorem~\ref{thm:metric-model}.
\end{proof}

\section{Active gaps and finite active obstructions}

The finite branch-defect condition is direct but still abstract. This section gives a concrete geometric sufficient condition in terms of gaps of $J=\supp\mu$.

\begin{definition}[Branchwise nonsingularity on the support]\label{def:branch-nonsingularity}
We say that $\mu$ is \emph{branchwise nonsingular on the support} if for every branch $I_k$ and every Borel set $N\subset J\cap I_k$ with $\mu(N)=0$, one has
\[
        \mu(\tau_k(N))=0.
\]
Equivalently, the translation $\tau_k:x\mapsto x+c_k$ sends $\mu$-null subsets of $J\cap I_k$ to $\mu$-null sets.
\end{definition}

\begin{remark}
This condition is not a substitute for invariance. It controls a different issue: thin subsets of the support. Invariance alone controls full preimages under $S$, while branchwise nonsingularity controls the direct image of null subsets inside a single branch. Without such a condition, a null subset of a singular support could in principle be translated onto a set of positive $\mu$-measure, producing a singular branch defect not visible from the relatively open gaps of $J$.
\end{remark}

\begin{definition}[Branchwise gaps and active gaps]\label{def:active-gap}
For a branch $I_k=[t_{k-1},t_k)$, write
\[
        I_k^\circ=(t_{k-1},t_k).
\]
A \emph{branchwise gap} in the branch $I_k$ is a connected component
\[
        G=(a,b)
\]
of $I_k^\circ\setminus J$. Thus branch endpoints are not included in branchwise gaps; their $h$-values are treated separately as ordinary branch cuts.

Since $G$ is contained in one branch, $S(G)=\tau_k(G)=G+c_k$ is an open interval. We say that $G$ is \emph{active} if
\[
        \mu(\tau_k(G))>0.
\]
The associated active-gap cut is the single value
\[
        h(G):=h(a)=h(b),
\]
where equality holds because $G\cap J=\emptyset$.

The active-gap cut set is
\[
        \cC_{\mathrm{gap}}
        =\{h(G):G\text{ is an active branchwise gap for some }k\}.
\]
\end{definition}

\begin{remark}
If $G=(a,b)\subset I_k^\circ\setminus J$, then $G$ is active exactly when $\tau_k(G)$ meets the support in positive $\mu$-measure. Because $\tau_k(G)$ is open and $J=\supp\mu$, this is equivalent to $\tau_k(G)\cap J\neq\emptyset$.
\end{remark}

\begin{lemma}[Gap-free intervals transport measure branchwise]\label{lem:branch-transport-away-gaps}
Assume branchwise nonsingularity on the support. Let $a<b$ be points of $X$ lying in the same branch $I_k$. Suppose that every branchwise gap $G$ contained in $(a,b)$ is inactive. Then
\[
        \mu((a+c_k,b+c_k))=\mu((a,b)).
\]
Consequently
\[
        \Delta_k(a)=\Delta_k(b),
\]
up to endpoint conventions, which are irrelevant because $\mu$ is non-atomic.
\end{lemma}

\begin{proof}
Let
\[
        A=(a,b)\cap J,
        \qquad
        U=(a,b)\setminus J.
\]
The open set $U$ is a countable disjoint union of branchwise gaps $G_r$ contained in $(a,b)$. By hypothesis every such $G_r$ is inactive, so
\[
        \mu(\tau_k(G_r))=0
\]
for every $r$. Hence, by countable subadditivity,
\[
        \mu(\tau_k(U))=0.
\]

We now treat the support part. Let $X$ be the invariant full-measure set from Definition~\ref{def:induced-T}. Since $A\setminus X$ is a $\mu$-null subset of $J\cap I_k$, branchwise nonsingularity gives
\[
        \mu(\tau_k(A\setminus X))=0.
\]
On $A\cap X$, the branch translation is exactly $S$, and $S:X\to X$ is a measure-preserving bijection. Therefore, by Corollary~\ref{cor:direct-core},
\[
        \mu(\tau_k(A\cap X))=\mu(S(A\cap X))=\mu(A\cap X)=\mu(A).
\]
Combining these two observations,
\[
        \mu(\tau_k(A))=\mu(A).
\]

Because $\tau_k$ is a translation on the whole branch,
\[
        (a+c_k,b+c_k)=\tau_k((a,b))=\tau_k(A)\cup\tau_k(U).
\]
The union need not be disjoint, but $\tau_k(U)$ has $\mu$-measure zero, and $\mu(\tau_k(A))=\mu(A)$. Thus
\[
        \mu((a+c_k,b+c_k))=\mu(A)=\mu((a,b)),
\]
since $U$ is disjoint from $J$ and hence has $\mu$-measure zero.

Finally, since $\mu$ has no atoms,
\[
\begin{aligned}
\Delta_k(b)-\Delta_k(a)
&=\bigl(h(b+c_k)-h(a+c_k)\bigr)-\bigl(h(b)-h(a)\bigr)\\
&=\mu((a+c_k,b+c_k))-\mu((a,b))=0.
\end{aligned}
\]
Hence $\Delta_k(a)=\Delta_k(b)$.
\end{proof}

\begin{lemma}[Avoiding branch cuts keeps one branch]\label{lem:avoid-branch-cuts}
Let $\cC$ be a set containing $h(t_j)$ for every branch endpoint $t_j$. Let $U$ be a connected component of $[0,1]\setminus\cC$. If $u<v$ are points of $U\cap Y$ and $a=q(u)$, $b=q(v)$, then no branch endpoint lies in $(a,b)$. Consequently $a$ and $b$ lie in the same continuity branch.
\end{lemma}

\begin{proof}
Suppose, to the contrary, that some endpoint $t_j$ lies in $(a,b)$. Since $h$ is non-decreasing,
\[
        u=h(a)\le h(t_j)\le h(b)=v.
\]
The inequalities are in fact strict. If, for instance, $h(t_j)=h(a)$, then $\mu((a,t_j))=0$. Hence $(a,t_j)\cap J=\emptyset$; otherwise an open neighborhood of a support point contained in $(a,t_j)$ would have positive $\mu$-measure. Thus $a$ would be an endpoint of a gap of $J$, or $a$ itself would be a branch endpoint. Both possibilities are excluded by the definition of $X$, from which all branch endpoints and all gap endpoints, together with their full orbits, were removed. The case $h(t_j)=h(b)$ is identical. Hence
\[
        h(t_j)\in(u,v)\subset U,
\]
contradicting $h(t_j)\in\cC$ and $U\cap\cC=\emptyset$. Therefore no endpoint lies between $a$ and $b$, so $a$ and $b$ lie in the same branch.
\end{proof}

\begin{lemma}[Avoiding active-gap cuts removes active gaps]\label{lem:avoid-active-cuts}
Let $\cC$ contain $\cC_{\mathrm{gap}}$, and let $U$ be a connected component of $[0,1]\setminus\cC$. If $u<v$ are points of $U\cap Y$ and $a=q(u)$, $b=q(v)$, then no active branchwise gap is contained in $(a,b)$.
\end{lemma}

\begin{proof}
Suppose that an active branchwise gap $G=(r,s)$ is contained in $(a,b)$. Since $h$ is non-decreasing,
\[
        u=h(a)\le h(r)=h(s)=h(G)\le h(b)=v.
\]
Again the inequalities are strict. For example, if $h(r)=h(a)$, then $\mu((a,r))=0$, so $(a,r)\cap J=\emptyset$; hence $a$ would be a gap endpoint, contrary to the construction of $X$. The other endpoint equalities are handled in the same way. Hence $h(G)\in(u,v)\subset U$. But $h(G)\in\cC_{\mathrm{gap}}\subset\cC$, contradicting $U\cap\cC=\emptyset$.
\end{proof}

\begin{theorem}[Finite active-gap criterion]\label{thm:finite-active-gap}
Assume that $\mu$ is branchwise nonsingular on the support. If the active-gap cut set $\cC_{\mathrm{gap}}$ is finite, then $(J,\mu,S)$ is metrically equivalent to a finite IET.
\end{theorem}

\begin{proof}
Let
\[
        \cC=\{h(t_0),h(t_1),\ldots,h(t_n)\}\cup\cC_{\mathrm{gap}}.
\]
By hypothesis this set is finite.

Let $U$ be a connected component of $[0,1]\setminus\cC$. We show that the finite branch-defect condition holds on $U$. Take $u<v$ in $U\cap Y$ and put
\[
        a=q(u),\qquad b=q(v).
\]
By Lemma~\ref{lem:avoid-branch-cuts}, the points $a$ and $b$ lie in a single continuity branch, say $I_k$. By Lemma~\ref{lem:avoid-active-cuts}, no active branchwise gap is contained in $(a,b)$. Thus every branchwise gap contained in $(a,b)$ is inactive. Lemma~\ref{lem:branch-transport-away-gaps} gives
\[
        \Delta_k(a)=\Delta_k(b).
\]
Since $u,v\in U\cap Y$ were arbitrary, $\Delta_k$ is constant on $q(U\cap Y)$. Therefore the finite branch-defect condition holds with the finite cut set $\cC$.

The conclusion follows from Theorem~\ref{thm:finite-defect}.
\end{proof}

\section{Defect measures and IET models}\label{sec:defect-measures}

The criterion in the previous section has a clear geometric meaning: active gaps
record those relatively open gaps of the support whose images acquire positive measure. However,
there is another possible obstruction which is less visible geometrically. A null set
may lie inside the support itself, and a branch translation may send this null set to a
set of positive measure. This is precisely the reason why branchwise nonsingularity was
assumed in Theorem~\ref{thm:finite-active-gap}.

In this section we package both phenomena into a single object. This is an intrinsic way
of measuring the total obstruction. Instead of separating obstructions arising from relatively open gaps from hidden
null-set obstructions, we measure directly the excess measure created by each branch
translation. The support of this defect measure, after passing to the distribution coordinate,
gives a natural cut set. If this cut set is finite, we obtain a finite IET model. If it is merely
Lebesgue-null, we obtain a countable IET model.

For each branch \(I_k\), write
\[
        \tau_k(x)=x+c_k.
\]
Define a finite Borel measure \(\eta_k\) on \(I_k\) by
\[
        \eta_k(B)=\mu(\tau_k(B)),
        \qquad B\subset I_k \text{ Borel},
\]
and put \(\mu_k=\mu|_{I_k}\). We define
\[
        \sigma_k=\eta_k-\mu_k.
\]
A priori this is a signed measure, but invariance forces it to be positive.

\begin{lemma}[Positivity of the branch defect]\label{lem:positive-branch-defect}
For every branch \(I_k\) and every Borel set \(B\subset I_k\),
\[
        \mu(\tau_k(B))\geq\mu(B).
\]
Consequently, \(\sigma_k\) is a finite positive Borel measure.
\end{lemma}

\begin{proof}
On \(I_k\), one has \(S(B)=\tau_k(B)\), and
\[
        B\subset S^{-1}(S(B)).
\]
By \(S\)-invariance,
\[
\mu(\tau_k(B))
=\mu(S(B))
=\mu(S^{-1}(S(B)))
\geq\mu(B).
\]
Thus \(\sigma_k(B)\geq0\) for every Borel \(B\subset I_k\).
\end{proof}

The measure \(\sigma_k\) records the excess measure carried by a direct branch image; no
negative defect or cancellation can occur. Push it to the distribution coordinate by setting
\[
        \rho_k=h_*\sigma_k,
\]
so that, for Borel sets \(E\subset[0,1]\),
\[
        \rho_k(E)=\sigma_k\bigl(h^{-1}(E)\cap I_k\bigr).
\]
Thus \(\rho_k\) records where the \(k\)-th branch displacement can change.

Finally, define the \emph{canonical defect cut set} by
\[
        \cC_{\mathrm{def}}
        =
        \{h(t_0),h(t_1),\ldots,h(t_n)\}
        \cup
        \bigcup_{k=1}^n \supp\rho_k.
\]
The first part records the ordinary branch cuts, and the second records the points in the
distribution coordinate where direct branch transport creates positive excess measure. We
say that the finite defect-support condition holds if \(\cC_{\mathrm{def}}\) is finite.

The following elementary identity explains why this is the right object.

\begin{lemma}[Defect-measure identity]\label{lem:defect-measure-identity}
Let \(a<b\) be two points lying in the same branch \(I_k\). Then, up to endpoint
conventions which are irrelevant because \(\mu\) is non-atomic,
\[
        \Delta_k(b)-\Delta_k(a)=\sigma_k((a,b)).
\]
\end{lemma}

\begin{proof}
Recall that
\[
        \Delta_k(x)=h(x+c_k)-h(x).
\]
Therefore
\[
\begin{aligned}
\Delta_k(b)-\Delta_k(a)
&=
\bigl(h(b+c_k)-h(b)\bigr)
-
\bigl(h(a+c_k)-h(a)\bigr) \\
&=
\bigl(h(b+c_k)-h(a+c_k)\bigr)
-
\bigl(h(b)-h(a)\bigr).
\end{aligned}
\]
Since \(\mu\) has no atoms,
\[
        h(b)-h(a)=\mu((a,b))
\]
and
\[
        h(b+c_k)-h(a+c_k)=\mu((a+c_k,b+c_k)).
\]
But
\[
        (a+c_k,b+c_k)=\tau_k((a,b)).
\]
Hence
\[
\Delta_k(b)-\Delta_k(a)
=
\mu(\tau_k((a,b)))-\mu((a,b))
=
\sigma_k((a,b)).
\]
\end{proof}

Before proving the finite and countable consequences, we record the local statement behind
the construction. The canonical defect cut set is sufficient: away from it, the
distribution-coordinate model is locally a translation. We do not claim that it is always the
smallest possible essential cut set.

\begin{theorem}[Defect-set structure theorem]\label{thm:defect-set-structure}\label{thm:finite-defect-support}\label{thm:null-defect-support}
Let \(T=h\circ S\circ q\) be the induced Lebesgue-preserving interval map. Then the defect
cut set \(\cC_{\mathrm{def}}\) has the following properties.
\begin{enumerate}[label=\textup{(\roman*)}]
    \item Let \(U\) be a connected component of
    \[
            [0,1]\setminus \cC_{\mathrm{def}}.
    \]
    If \(U\cap Y\neq\emptyset\), then \(q(U\cap Y)\) is contained in a single branch
    \(I_k\), and there is a constant \(d_U\) such that
    \[
            T(u)=u+d_U
    \]
    for every \(u\in U\cap Y\). In particular, \(T\) is a translation on \(U\) modulo
    Lebesgue-null sets.

    \item If \(\cC_{\mathrm{def}}\) is finite, then \((J,\mu,S)\) is metrically equivalent
    to a finite IET.

    \item If
    \[
            \Leb(\cC_{\mathrm{def}})=0,
    \]
    then \((J,\mu,S)\) is metrically equivalent to a countable IET.
\end{enumerate}
\end{theorem}

\begin{proof}
We first prove the local statement in (i). Let \(U\) be a connected component of
\([0,1]\setminus \cC_{\mathrm{def}}\). If \(U\cap Y\) is empty or consists of only one point,
there is nothing to prove. Otherwise take two points
\[
        u<v, \qquad u,v\in U\cap Y,
\]
and put
\[
        a=q(u),\qquad b=q(v).
\]
Since \(h:X\to Y\) is one-to-one and non-decreasing, its inverse \(q:Y\to X\) is
order-preserving. Hence
\[
        a<b,\qquad h(a)=u,\qquad h(b)=v.
\]

The set \(\cC_{\mathrm{def}}\) contains all branch endpoint values
\[
        h(t_0),h(t_1),\ldots,h(t_n).
\]
Therefore Lemma~\ref{lem:avoid-branch-cuts} shows that no branch endpoint lies between
\(a\) and \(b\). Thus \(a\) and \(b\) lie in one continuity branch, say \(I_k\). Since this is
true for every pair of points in \(U\cap Y\), the whole set \(q(U\cap Y)\) is contained in a
single branch \(I_k\).

Now let \(x\in(a,b)\). By monotonicity of \(h\),
\[
        u=h(a)\leq h(x)\leq h(b)=v.
\]
Because \(U\) is an interval and \(u,v\in U\), it follows that
\[
        h(x)\in U.
\]
Hence
\[
        (a,b)\subset h^{-1}(U)\cap I_k.
\]
On the other hand, by the definition of \(\cC_{\mathrm{def}}\),
\[
        U\cap \supp\rho_k=\emptyset.
\]
Therefore
\[
        \rho_k(U)=0.
\]
Using the definition of \(\rho_k\), we get
\[
        0=\rho_k(U)=\sigma_k\bigl(h^{-1}(U)\cap I_k\bigr).
\]
Since \((a,b)\subset h^{-1}(U)\cap I_k\), this gives
\[
        \sigma_k((a,b))=0.
\]
By Lemma~\ref{lem:defect-measure-identity},
\[
        \Delta_k(b)-\Delta_k(a)=0.
\]
Thus
\[
        \Delta_k(a)=\Delta_k(b).
\]
Since the points \(u,v\in U\cap Y\) were arbitrary, \(\Delta_k\) is constant on
\(q(U\cap Y)\). Denote this constant by \(d_U\). Then for every \(u\in U\cap Y\), with
\(x=q(u)\), we have
\[
        T(u)=h(Sx)=h(x+c_k)=h(x)+\Delta_k(x)=u+d_U.
\]
This proves (i).

If \(\cC_{\mathrm{def}}\) is finite, then \([0,1]\setminus\cC_{\mathrm{def}}\) has only finitely
many connected components. By (i), on each component the induced map \(T\) agrees almost
everywhere with a translation. Hence the finite branch-defect condition of
Theorem~\ref{thm:finite-defect} holds. That theorem then implies that \((J,\mu,S)\) is
metrically equivalent to a finite IET. This proves (ii).

It remains to prove (iii). The set \(\cC_{\mathrm{def}}\) is closed, because it is a finite
union of closed supports \(\supp\rho_k\), together with finitely many branch endpoint values.
Also \(0=h(t_0)\) and \(1=h(t_n)\) belong to \(\cC_{\mathrm{def}}\). Therefore
\[
        [0,1]\setminus \cC_{\mathrm{def}}
\]
is a countable disjoint union of open intervals. Write
\[
        [0,1]\setminus \cC_{\mathrm{def}}=\bigcup_{m=1}^{\infty} U_m.
\]
If \(\Leb(\cC_{\mathrm{def}})=0\), then these intervals cover \([0,1]\) modulo a Lebesgue-null
set. By (i), for each \(m\) there is a constant \(d_m\) such that
\[
        T(u)=u+d_m
\]
for almost every \(u\in U_m\). Since the family \(\{U_m\}\) is countable and \(Y\) has
full Lebesgue measure, we may redefine \(T\) on the null set
\[
        \cC_{\mathrm{def}}\cup\bigcup_{m\geq1}(U_m\setminus Y)
\]
so that \(T(u)=u+d_m\) for every \(u\in U_m\). This modification does not change the
measured system or Lebesgue preservation, because it is confined to a Lebesgue-null set.
Moreover,
\[
        U_m+d_m\subseteq[0,1]
\]
up to endpoints. Indeed, otherwise a nonempty open subinterval of \(U_m\) would be
translated outside \([0,1]\), contradicting that the original map \(T\) is
\([0,1]\)-valued on the full-measure set \(U_m\cap Y\). Thus \(T\) is represented by a
countable piecewise translation on the intervals \(U_m\).

We now check that this representative is a countable IET.
Put
\[
        V_m=U_m+d_m.
\]
Suppose that two distinct image intervals \(V_i\) and \(V_j\) overlap in a Borel set \(B\)
of positive Lebesgue measure. Since \(T\) is a translation on \(U_i\), the preimage of \(B\)
inside \(U_i\) has measure \(\Leb(B)\). Similarly, the preimage of \(B\) inside \(U_j\) also
has measure \(\Leb(B)\). These two preimage sets are disjoint, and therefore
\[
        \Leb(T^{-1}B)\geq 2\Leb(B).
\]
This contradicts the Lebesgue preservation of \(T\), proved in Theorem~\ref{thm:metric-model}.
Hence the image intervals \(V_m\) are pairwise disjoint modulo null sets.

Finally, translations preserve Lebesgue measure, so
\[
        \sum_m \Leb(V_m)=\sum_m \Leb(U_m)=1.
\]
Since the intervals \(V_m\) are pairwise disjoint modulo null sets, their union has full
Lebesgue measure. Thus \(T\) is a countable IET modulo null
sets. The metric equivalence with \((J,\mu,S)\) follows again from Theorem~\ref{thm:metric-model}.
This proves (iii).
\end{proof}

The next result shows that finite defect support is not merely another finite-cut
condition. It also rules out the hidden singular obstruction which branchwise
nonsingularity was designed to exclude.

\begin{proposition}[Finite defect support implies branchwise nonsingularity]\label{prop:defect-support-nonsingular}
Suppose that \(\supp\rho_k\) is finite for a branch \(I_k\). Then \(\mu\) is branchwise
nonsingular on \(I_k\). That is, if
\[
        N\subset J\cap I_k
\]
is Borel and
\[
        \mu(N)=0,
\]
then
\[
        \mu(\tau_k(N))=0.
\]
Consequently, if \(\cC_{\mathrm{def}}\) is finite, then \(\mu\) is branchwise
nonsingular on the support.
\end{proposition}

\begin{proof}
Suppose, toward a contradiction, that branchwise nonsingularity fails on \(I_k\). Then
there exists a Borel set
\[
        N\subset J\cap I_k
\]
such that
\[
        \mu(N)=0
\]
but
\[
        \mu(\tau_k(N))>0.
\]
By definition,
\[
        \eta_k(N)=\mu(\tau_k(N))>0,
\]
whereas
\[
        \mu_k(N)=\mu(N)=0.
\]
Thus
\[
        \sigma_k(N)>0.
\]

Let
\[
        F=\supp\rho_k.
\]
By assumption \(F\) is finite. Since \(\rho_k\) is supported on \(F\),
\[
        \rho_k([0,1]\setminus F)=0.
\]
Using the definition of \(\rho_k\), this gives
\[
        \sigma_k\bigl(h^{-1}([0,1]\setminus F)\cap I_k\bigr)=0.
\]
Since \(\sigma_k(N)>0\), it follows that
\[
        \sigma_k(N\cap h^{-1}(F))>0.
\]

We now show that this is impossible. For each \(r\in[0,1]\), the fiber \(h^{-1}(r)\)
is a closed interval, possibly degenerate, because \(h\) is continuous and
non-decreasing. If
\[
        h^{-1}(r)=[p,q]
\]
with \(p<q\), then \(h\) is constant on \((p,q)\), and hence
\[
        \mu((p,q))=0.
\]
Since \(J=\supp\mu\), no interior point of \((p,q)\) can belong to \(J\). Thus
\[
        J\cap h^{-1}(r)\subset\{p,q\}.
\]
If the fiber is degenerate, then \(J\cap h^{-1}(r)\) contains at most one point.
Therefore, for every \(r\),
\[
        J\cap h^{-1}(r)
\]
is finite.

Since \(F\) is finite, we conclude that
\[
        J\cap h^{-1}(F)
\]
is finite. Hence
\[
        N\cap h^{-1}(F)
\]
is finite.

Because \(\mu\) is non-atomic,
\[
        \mu(N\cap h^{-1}(F))=0.
\]
Also, \(\tau_k(N\cap h^{-1}(F))\) is finite, so again by non-atomicity,
\[
        \mu(\tau_k(N\cap h^{-1}(F)))=0.
\]
Therefore both \(\mu_k\) and \(\eta_k\) vanish on \(N\cap h^{-1}(F)\), and hence
\[
        \sigma_k(N\cap h^{-1}(F))=0.
\]
This contradicts
\[
        \sigma_k(N\cap h^{-1}(F))>0.
\]
Thus branchwise nonsingularity cannot fail on \(I_k\).
\end{proof}

We also record the relation between this canonical defect cut set and the active gaps from
the previous section.

For a fixed branch \(I_k\), write
\[
        \cC_{\mathrm{gap},k}
        =
        \{h(G):G \text{ is an active branchwise gap in } I_k\}.
\]
Thus
\[
        \cC_{\mathrm{gap}}=\bigcup_{k=1}^n \cC_{\mathrm{gap},k}.
\]

\begin{lemma}[Active gaps give defect atoms]\label{lem:active-gaps-defect-atoms}
Let \(G=(a,b)\) be an active branchwise gap contained in the branch \(I_k\). Then
\[
        h(G)\in \supp\rho_k.
\]
In particular,
\[
        \cC_{\mathrm{gap}}
        \subset
        \bigcup_{k=1}^n \supp \rho_k.
\]
\end{lemma}

\begin{proof}
Since \(G\) is a gap of \(J\),
\[
        \mu(G)=0.
\]
Since \(G\) is active,
\[
        \mu(\tau_k(G))>0.
\]
Therefore
\[
        \sigma_k(G)=\mu(\tau_k(G))-\mu(G)>0.
\]

On the other hand, \(h\) is constant on \(G\), because \(G\cap J=\emptyset\). Thus
\[
        h(G)=h(a)=h(b).
\]
Since
\[
        G\subset h^{-1}(\{h(G)\}),
\]
we get
\[
\rho_k(\{h(G)\})
=
\sigma_k\bigl(h^{-1}(\{h(G)\})\cap I_k\bigr)
\ge
\sigma_k(G)
>
0.
\]
Therefore
\[
        h(G)\in \supp\rho_k.
\]
\end{proof}

The previous lemma shows that active-gap cuts belong to the defect support. Under
branchwise nonsingularity, there are no other defect points: the defect support is exactly
the closure of the active-gap cuts.

\begin{proposition}[Defect support and active-gap cuts]\label{prop:defect-support-active-closure}
Assume that \(\mu\) is branchwise nonsingular on the support. Then, for every branch
\(I_k\),
\[
        \supp\rho_k=\overline{\cC_{\mathrm{gap},k}}.
\]
\end{proposition}

\begin{proof}
Let \(U\subset[0,1]\) be an open interval such that
\[
        U\cap \overline{\cC_{\mathrm{gap},k}}=\emptyset.
\]
We prove that \(\rho_k(U)=0\). This will show that no point outside
\(\overline{\cC_{\mathrm{gap},k}}\) can belong to \(\supp\rho_k\).

Put
\[
        E=h^{-1}(U)\cap I_k.
\]
It is enough to show that \(\sigma_k(E)=0\). We prove the stronger statement that \(\sigma_k(B)=0\) for every Borel set \(B\subset E\).

Fix such a Borel set \(B\). Split it into the part on the support and the part outside the
support:
\[
        B_J=B\cap J,
        \qquad
        B_G=B\setminus J.
\]
First consider \(B_J\). We further decompose
\[
        B_J=(B_J\cap X)\cup (B_J\setminus X).
\]
Since \(X\) has full \(\mu\)-measure in \(J\), we have
\[
        \mu(B_J\setminus X)=0.
\]
Also \(B_J\setminus X\subset J\cap I_k\). By branchwise nonsingularity,
\[
        \mu(\tau_k(B_J\setminus X))=0.
\]
On \(B_J\cap X\), the branch translation is exactly the map \(S\). Since
\(S:X\to X\) is a measure-preserving bijection, Corollary~\ref{cor:direct-core}
gives
\[
        \mu(\tau_k(B_J\cap X))
        =
        \mu(S(B_J\cap X))
        =
        \mu(B_J\cap X).
\]
Adding the null part does not change the measure, hence
\[
        \mu(\tau_k(B_J))=\mu(B_J).
\]

Now consider \(B_G\). Since \(B_G\cap J=\emptyset\), we have \(\mu(B_G)=0\). We claim
that also
\[
        \mu(\tau_k(B_G))=0.
\]
Indeed, the set \(I_k^\circ\setminus J\) is a countable union of branchwise gaps, and
\(B_G\) is contained in the union of these gaps together with at most finitely many
branch endpoints. The branch endpoints create no measure after translation, because
\(\mu\) is non-atomic and their images are finite sets. Hence it is enough to consider
those branchwise gaps \(G\subset I_k^\circ\setminus J\) which meet \(B_G\).
For such a gap, \(h\) is constant on \(G\), and since \(G\cap B_G\subset E\), this
constant value lies in \(U\). Thus
\[
        h(G)\in U.
\]
But \(U\cap\cC_{\mathrm{gap},k}=\emptyset\), so \(G\) is not active. Hence
\[
        \mu(\tau_k(G))=0.
\]
By countable subadditivity, including the finite endpoint part if present,
\[
        \mu(\tau_k(B_G))=0.
\]
Therefore
\[
        \mu(\tau_k(B_G))=\mu(B_G)=0.
\]

Finally,
\[
        B=B_J\cup B_G.
\]
The image of \(B_G\) has zero \(\mu\)-measure, and the image of \(B_J\) has measure
\(\mu(B_J)\). Hence
\[
        \mu(\tau_k(B))=\mu(B_J)=\mu(B).
\]
Therefore
\[
        \sigma_k(B)=0
\]
for every Borel subset \(B\subset E\). It follows that \(\sigma_k(E)=0\), and so
\[
        \rho_k(U)=\sigma_k(h^{-1}(U)\cap I_k)=0.
\]
Thus
\[
        \supp\rho_k\subset \overline{\cC_{\mathrm{gap},k}}.
\]
Conversely, Lemma~\ref{lem:active-gaps-defect-atoms} gives
\[
        \cC_{\mathrm{gap},k}\subset \supp\rho_k.
\]
Since \(\supp\rho_k\) is closed, it also contains the closure of
\(\cC_{\mathrm{gap},k}\). Therefore
\[
        \overline{\cC_{\mathrm{gap},k}}\subset \supp\rho_k,
\]
and the two inclusions prove the equality.
\end{proof}

We can now state the exact relation between the geometric criterion of the previous section
and the intrinsic defect-support condition.

\begin{corollary}[Equivalence of finite obstruction conditions]\label{cor:finite-obstruction-equivalence}
The following are equivalent:
\begin{enumerate}[label=\textup{(\roman*)}]
    \item \(\cC_{\mathrm{def}}\) is finite;
    \item \(\mu\) is branchwise nonsingular on the support and \(\cC_{\mathrm{gap}}\) is finite.
\end{enumerate}
\end{corollary}

\begin{proof}
Assume first that \(\cC_{\mathrm{def}}\) is finite. Then every \(\supp\rho_k\) is finite.
By Proposition~\ref{prop:defect-support-nonsingular}, \(\mu\) is branchwise nonsingular
on every branch. Also, by Lemma~\ref{lem:active-gaps-defect-atoms},
\[
        \cC_{\mathrm{gap}}
        \subset
        \bigcup_{k=1}^n \supp\rho_k.
\]
The right-hand side is finite, so \(\cC_{\mathrm{gap}}\) is finite.

Conversely, assume that \(\mu\) is branchwise nonsingular on the support and that
\(\cC_{\mathrm{gap}}\) is finite. Then each \(\cC_{\mathrm{gap},k}\) is finite and
therefore closed. Hence
\[
        \overline{\cC_{\mathrm{gap},k}}=\cC_{\mathrm{gap},k}.
\]
By Proposition~\ref{prop:defect-support-active-closure},
\[
        \supp\rho_k
        =\overline{\cC_{\mathrm{gap},k}}
        =\cC_{\mathrm{gap},k}
\]
for every \(k\). Hence each \(\supp\rho_k\) is finite. Since there are only finitely many
branch endpoint values \(h(t_j)\), the set \(\cC_{\mathrm{def}}\) is finite.
\end{proof}

\begin{remark}
The finite defect-support condition is therefore not merely an additional sufficient
hypothesis. It is the intrinsic measure-theoretic form of the two finite obstructions used
in the active-gap criterion. Active gaps are the visible geometric part of the defect, while
branchwise nonsingularity rules out hidden singular defects inside the support. The defect
measures combine both effects into a single object.
\end{remark}

\begin{remark}
The conclusions of Theorem~\ref{thm:defect-set-structure} are metric, not topological.
Part~\textup{(ii)} does not assert that the original interval translation map is of finite
type, nor that the map literally becomes an IET on a geometric attractor.
It asserts that, after passing to the distribution coordinate and ignoring null sets, the
measured system is represented by a finite IET. Part~\textup{(iii)} gives
the analogous countable statement when the canonical defect cut set is null but not necessarily
finite.
\end{remark}

\section{Examples, basic cases, and geometric reformulations}

The preceding criteria are most useful when they can be checked directly from the support or
from a familiar invariant measure. We record a few basic cases and then collect several
equivalent ways of detecting active gaps.

\begin{example}[The IET case]\label{ex:iem}
Suppose that \(S\) is already a finite IET and
\(\mu=\Leb\). Then \(J=[0,1]\), \(h(x)=x\), and every branch translation preserves
Lebesgue measure. Hence \(\sigma_k=0\) for every \(k\),
\[
        \cC_{\mathrm{def}}=\{t_0,\ldots,t_n\},
\]
and the induced map in the distribution coordinate is exactly \(S\).
\end{example}

\begin{example}[A finite-type map with an explicit two-IET model]
\label{ex:explicit-finite-type}
Partition $I=[0,1)$ into four equal branches and define
\[
S(x)=
\begin{cases}
 x+\dfrac12, & 0\le x<\dfrac14,\\[2mm]
 x-\dfrac14, & \dfrac14\le x<\dfrac12,\\[2mm]
 x-\dfrac12, & \dfrac12\le x<\dfrac34,\\[2mm]
 x-\dfrac14, & \dfrac34\le x<1.
\end{cases}
\]
Then
\[
S(I)=
\left[0,\frac14\right)
\cup
\left[\frac12,\frac34\right),
\]
and this set is invariant under $S$. Put
\[
J=\left[0,\frac14\right]
\cup
\left[\frac12,\frac34\right]
\]
and let
\[
\mu=2\,\Leb|_J.
\]
Thus $\mu$ is the normalized Lebesgue measure on $J$. On the two components of $J$, the
map exchanges the intervals by the translations $x\mapsto x+\frac12$ and
$x\mapsto x-\frac12$. Hence $\mu$ is $S$-invariant.

The distribution function is
\[
h(x)=
\begin{cases}
 2x, & 0\le x\le \dfrac14,\\[1mm]
 \dfrac12, & \dfrac14\le x\le \dfrac12,\\[1mm]
 2x-\dfrac12, & \dfrac12\le x\le \dfrac34,\\[1mm]
 1, & \dfrac34\le x\le 1.
\end{cases}
\]
The first and third branches transport $\mu$ exactly, and therefore
\[
\sigma_1=\sigma_3=0.
\]
The second and fourth branches lie outside the support, but their images are the two
components of $J$. More precisely,
\[
\sigma_2=2\,\Leb|_{[1/4,1/2)},
\qquad
\sigma_4=2\,\Leb|_{[3/4,1)}.
\]
Since $h$ is constant on these two branches,
\[
\rho_2=\frac12\,\delta_{1/2},
\qquad
\rho_4=\frac12\,\delta_1.
\]
The branch endpoint values are
\[
h(0)=0,
\qquad
h\left(\frac14\right)=h\left(\frac12\right)=\frac12,
\qquad
h\left(\frac34\right)=h(1)=1.
\]
Consequently,
\[
\cC_{\mathrm{def}}=\left\{0,\frac12,1\right\}.
\]
The two gaps
\[
\left(\frac14,\frac12\right)
\quad\text{and}\quad
\left(\frac34,1\right)
\]
are active, with cut values $\frac12$ and $1$, respectively.

The distribution coordinate maps the two components of $J$ onto the two half-intervals
$[0,\frac12]$ and $[\frac12,1]$. The induced map is
\[
T(u)=
\begin{cases}
 u+\dfrac12, & 0\le u<\dfrac12,\\[2mm]
 u-\dfrac12, & \dfrac12\le u<1,
\end{cases}
\]
modulo endpoints. Thus the canonical model is the two-interval IET given by rotation
through one half-turn.
\end{example}

\subsection{A self-similar infinite-type example with an explicit canonical defect cut set}
\label{subsec:BT-explicit-defect}

We now compute the canonical defect cut set for a stationary member of the
Bruin--Troubetzkoy family. This gives an explicit infinite-type ITM whose ordinary
distribution-coordinate model is a genuinely countable, non-finite IET.

The external symbolic input from Bruin--Troubetzkoy is the following. Their
Theorem~7 and its proof describe the inducing substitutions, identify the
infinite-type dynamical attractor with the corresponding one-sided substitution
shift, and show that the future-itinerary coding is one-to-one except
possibly on the backward orbit of the two discontinuity points
\cite[Theorem~7, pp.~130 and 135]{BruinTroubetzkoy2003}.
For the stationary parameter considered below, their substitution chain is
constant and equal to \(\chi_2\). Further developments of the renormalization theory for
the Bruin--Troubetzkoy family appear in
\cite{ArtigianiHubertSkripchenko2026}.

The exact compact graph-directed equations, the complete enumeration of
the complementary gap towers, the identification of all active gaps, the
explicit cut values and the essential non-finiteness argument are proved in
the present paper.

We retain the interval convention of Section~2 throughout this subsection. The dynamical
domain is the half-open interval
\[
        I=[0,1),
\]
and all iterates of \(S\) are formed on \(I\). For topological statements we work in the
compactification \(\overline I=[0,1]\) and use the compact attractor
\[
        J_{\mathrm{att}}
        :=
        J_{\mathrm{att}}(S)
        =
        \bigcap_{m\geq0}\overline{S^m(I)}^{\,[0,1]}.
\]
We write
\[
        J_{\mathrm{dyn}}:=J_{\mathrm{att}}\cap I
\]
for the actual dynamical space. The map \(S\) is applied only on
\(J_{\mathrm{dyn}}\); whenever endpoints of tower intervals are followed, the prescribed
branch translations are replaced by their continuous affine extensions to the corresponding
closed branch intervals.
Below we identify the unique invariant probability and prove that its support, viewed in
\([0,1]\), is exactly \(J_{\mathrm{att}}\).

Let \(\alpha\) be the unique root in \((0,\frac12)\) of
\[
        P(x)=x^3-x^2-2x+1,
\]
and put
\[
        \beta=\alpha^2,
        \qquad
        s=1-\alpha,
        \qquad
        q=s^2.
\]
Numerically,
\[
        \alpha\approx0.4450418679,
        \qquad
        \beta\approx0.1980622642,
        \qquad
        q\approx0.3079785284.
\]
Consider
\[
S(x)=
\begin{cases}
 x+\alpha, & x\in I_1=[0,1-\alpha),\\[1mm]
 x+\beta, & x\in I_2=[1-\alpha,1-\beta),\\[1mm]
 x+\beta-1, & x\in I_3=[1-\beta,1).
\end{cases}
\tag{8.1}
\label{eq:BT-map}
\]

\subsubsection{Stationary inducing and the invariant measure}

Let
\[
        \Delta=[1-\alpha,1)
        \qquad\text{and}\qquad
        \phi:I\longrightarrow\Delta,
        \quad
        \phi(x)=1-\alpha+\alpha x.
\]
We use the same symbol \(\phi\) for its continuous extension
\([0,1]\to[1-\alpha,1]\) when endpoints or closures are considered.
The renormalization map for this family is
\[
G(a,b)=
\left(
\frac{b}{a},
\frac{b-1}{a}+\left\lfloor\frac1a\right\rfloor
\right).
\]
Since \(\lfloor1/\alpha\rfloor=2\), \(\beta=\alpha^2\), and
\[
        \alpha^3=\alpha^2+2\alpha-1,
\]
we have
\[
        G(\alpha,\beta)=(\alpha,\beta).
\]
Thus the parameter is an interior fixed point of the renormalization map. By
\cite[Corollary~5]{BruinTroubetzkoy2003}, the map \eqref{eq:BT-map} is of infinite type and
its compact attractor
\[
        J_{\mathrm{att}}
        =
        J_{\mathrm{att}}(S)
        =
        \bigcap_{m\geq0}\overline{S^m(I)}^{\,[0,1]}
        \subset[0,1]
\]
is a Cantor set.  The point \(1\), when it belongs to this compactification, is a null
endpoint and is not part of the dynamical domain.

The first-return map to \(\Delta\) is conjugate to \(S\) through \(\phi\). The return words
are
\[
\chi:
\begin{cases}
1\longmapsto2,\\
2\longmapsto311,\\
3\longmapsto31.
\end{cases}
\tag{8.2}
\label{eq:BT-substitution}
\]
This is the substitution \(\chi_2\) from
\cite[Theorem~7]{BruinTroubetzkoy2003}. For completeness, if
\(D_i=\phi(I_i)\), then, modulo endpoints,
\[
\begin{aligned}
D_1&=[1-\alpha,1-\beta),\\
D_2&=[1-\beta,1-\alpha\beta),\\
D_3&=[1-\alpha\beta,1),
\end{aligned}
\]
and points in \(D_1,D_2,D_3\) return to \(\Delta\) after respectively
\(1,3,2\) iterates, with itineraries \(2,311,31\).

The incidence matrix of \(\chi\), with columns counting letters in the substituted words, is
\[
A=
\begin{pmatrix}
0&2&1\\
1&0&0\\
0&1&1
\end{pmatrix}.
\tag{8.3}
\label{eq:BT-matrix}
\]
The matrix is primitive. Hence the substitution system is uniquely ergodic; see, for
example, \cite{Queffelec2010}. More precisely, the stationary Bratteli--Vershik model
associated with the substitution has a unique invariant probability measure determined by
the Perron--Frobenius eigenvector, by
\cite[Theorem~3.8 and Remark~3.9(1),(2)]{BezuglyiKwiatkowskiMedynetsSolomyak2010};
the correspondence between invariant measures of the substitution system and of its
stationary diagram is given in
\cite[Theorem~5.4]{BezuglyiKwiatkowskiMedynetsSolomyak2010}. Its normalized
Perron--Frobenius frequency vector is
\[
        p=
        \begin{pmatrix}
        \alpha\\
        \alpha-\alpha^2\\
        (1-\alpha)^2
        \end{pmatrix}.
\tag{8.4}
\label{eq:BT-frequency}
\]
Indeed, the Perron--Frobenius eigenvalue is \(1/(1-\alpha)\), and a direct calculation using
\(P(\alpha)=0\) gives
\[
        Ap=\frac{1}{1-\alpha}p,
        \qquad
        p_1+p_2+p_3=1.
\]

The following result is imported from Bruin--Troubetzkoy
\cite[Theorem~7 and its proof, pp.~130 and 135]{BruinTroubetzkoy2003}; we record its
specialization to the stationary parameter in the precise form needed below.

\begin{proposition}[Bruin--Troubetzkoy coding, specialized to the stationary parameter]
\label{lem:BT-symbolic-coding}
Let
\[
        \chi:
        \begin{cases}
        1\mapsto2,\\
        2\mapsto311,\\
        3\mapsto31,
        \end{cases}
\]
and let \(X_\chi\subset\{1,2,3\}^{\mathbb N_0}\) be the one-sided
substitution subshift generated by \(\chi\).

For \(x\in J_{\mathrm{dyn}}\), define its future branch itinerary by
\[
        \kappa(x)
        =
        \kappa_0(x)\kappa_1(x)\kappa_2(x)\cdots,
\]
where
\[
        \kappa_r(x)=i
        \quad\Longleftrightarrow\quad
        S^r(x)\in I_i.
\]
Then:

\begin{enumerate}[label=\textup{(\roman*)}]
\item
\(\kappa:J_{\mathrm{dyn}}\to X_\chi\) is Borel and satisfies
\[
        \kappa\circ S=\sigma\circ\kappa.
\]

\item
After adjoining the endpoint \(1\) and using the left-continuous endpoint
extension adopted by Bruin--Troubetzkoy, the coding extends to a continuous
surjection
\[
        \overline\kappa:J_{\mathrm{att}}\longrightarrow X_\chi.
\]

\item
The restriction of \(\kappa\) is one-to-one outside
\[
        E_0
        :=
        J_{\mathrm{dyn}}
        \cap
        \bigcup_{r\geq0}
        S^{-r}\{1-\alpha,1-\beta\}.
\]
In particular, \(E_0\) is a countable Borel set.
\end{enumerate}
\end{proposition}

\begin{proof}[Source and specialization]
Bruin--Troubetzkoy define, for every point of their dynamical attractor, the future
itinerary with respect to the three half-open branch intervals. Their
Theorem~7 identifies the restriction of the interval translation map to its
infinite-type dynamical attractor with the one-sided shift generated by the chain of
substitutions
\[
        \chi_{k_0}\circ\chi_{k_1}\circ\chi_{k_2}\circ\cdots,
\]
where
\[
        \chi_k:
        \begin{cases}
        1\mapsto2,\\
        2\mapsto31^k,\\
        3\mapsto31^{k-1}.
        \end{cases}
\]
See
\cite[Theorem~7, p.~130]{BruinTroubetzkoy2003}.

For the parameter considered here, the renormalization orbit is the fixed
point lying in the region \(U_2\). Hence
\[
        k_r=2
        \qquad
        \text{for every }r\geq0,
\]
and the substitution chain is stationary:
\[
        \chi_2\circ\chi_2\circ\cdots.
\]
Thus the shift space \(\Sigma_{\alpha,\beta}\) of
\cite{BruinTroubetzkoy2003} is precisely the stationary substitution
subshift \(X_\chi\) used here.

The proof of their Theorem~7 describes the inducing process explicitly.
If a point has itinerary \(u\) for the first-return map at inducing depth
\(n\), then its itinerary at depth \(n-1\) is obtained by applying
\(\chi_{k_{n-1}}\). Iterating this relation produces the complete future
branch itinerary. The same proof states that this itinerary construction
extends to a continuous coding of the compactified attractor and that the
coding is one-to-one except possibly on
\[
        \bigcup_{r\geq0}
        S^{-r}\{1-\alpha,1-\beta\};
\]
see
\cite[proof of Theorem~7, p.~135]{BruinTroubetzkoy2003}.

Bruin--Troubetzkoy define the value of the map at the added endpoint \(1\)
by its left limit. In this paper, \(1\) is retained only in the compact
attractor \(J_{\mathrm{att}}\), whereas the actual dynamical space is
\[
        J_{\mathrm{dyn}}=J_{\mathrm{att}}\cap[0,1).
\]
Restricting their compactified coding to \(J_{\mathrm{dyn}}\) therefore
gives the map \(\kappa\) above. It is Borel, and the identity
\[
        \kappa(Sx)=\sigma(\kappa(x))
\]
follows directly from the definition of the future itinerary.

Finally, each point has at most three preimages under \(S\). Hence
\(S^{-r}\{1-\alpha,1-\beta\}\) is finite for every \(r\), and therefore
\(E_0\) is countable. It is Borel because it is a countable union of finite
sets.
\end{proof}

We next justify that the interval system on
\[
        J_{\mathrm{dyn}}=J_{\mathrm{att}}\cap I
\]
has a unique invariant probability, including the possibility of atomic invariant measures.

We first record two elementary observations. If a measurable map \(F\) preserves a
probability measure \(\lambda\) and \(\lambda\) has an atom at \(x\), then \(x\) is periodic.
Indeed,
\[
        \lambda(\{F(x)\})\geq\lambda(\{x\}),
\]
so the masses along the forward orbit are nondecreasing. An infinite forward orbit would
contain infinitely many disjoint atoms of mass at least \(\lambda(\{x\})\), which is
impossible. Thus \(x\) is eventually periodic. If \(x\) entered a periodic cycle without
already belonging to it, invariance of the singleton at the entry point would count both its
predecessor on the cycle and a positive-mass predecessor outside the cycle, again a
contradiction. Hence every atom lies on a periodic orbit.

The substitution subshift \(X_\chi\) contains no periodic sequence. Indeed, a periodic
sequence would support an invariant probability whose frequency of the letter \(1\) is
rational. Unique ergodicity would force this measure to be the substitution measure, whose
frequency of \(1\) is \(\alpha\). Since \(\alpha\) is irrational by the rational-root theorem
applied to \(x^3-x^2-2x+1\), this is impossible.

Let
\[
        \kappa:J_{\mathrm{dyn}}\longrightarrow X_\chi
\]
be the future-itinerary coding of
Proposition~\ref{lem:BT-symbolic-coding}, and put
\[
        E_0
        =
        J_{\mathrm{dyn}}
        \cap
        \bigcup_{r\geq0}
        S^{-r}\{1-\alpha,1-\beta\}.
\]
By that lemma,
\[
        \kappa\circ S=\sigma\circ\kappa,
\]
and
\[
        \kappa|_{J_{\mathrm{dyn}}\setminus E_0}
\]
is one-to-one.

Put
\[
        Z:=J_{\mathrm{dyn}}\setminus E_0,
        \qquad
        \kappa_Z:=\kappa|_Z.
\]
The map
\[
        \kappa_Z:Z\longrightarrow X_\chi
\]
is a Borel injection between standard Borel spaces. By the
Lusin--Souslin theorem \cite{Kechris1995}, \(\kappa_Z(Z)\) is Borel and
\[
        \kappa_Z^{-1}:\kappa_Z(Z)\longrightarrow Z
\]
is Borel.

If \(x\in J_{\mathrm{dyn}}\) were periodic of period \(p\), then
\[
        \sigma^p(\kappa(x))
        =
        \kappa(S^p x)
        =
        \kappa(x).
\]
Thus \(\kappa(x)\) would be a periodic point of \(X_\chi\), contrary to the
absence of periodic sequences in the primitive substitution subshift.
Therefore \(S|_{J_{\mathrm{dyn}}}\) has no periodic points.

We now construct an invariant probability on \(J_{\mathrm{dyn}}\), rather
than merely proving that such a probability would be unique. Let
\[
        W:=\kappa_Z(Z)\subset X_\chi.
\]
The set \(W\) is Borel by the Lusin--Souslin theorem. The compactified
coding in Proposition~\ref{lem:BT-symbolic-coding} is onto \(X_\chi\), and
\[
        J_{\mathrm{att}}\setminus Z\subset E_0\cup\{1\}.
\]
Consequently,
\[
        X_\chi\setminus W
        \subset
        \kappa(E_0\cup\{1\}),
\]
which is countable. The substitution measure \(\nu_\chi\) is non-atomic:
indeed, an atom of an invariant probability for the one-sided shift must
lie on a periodic orbit, whereas \(X_\chi\) has no periodic sequence by
the argument above. Hence
\[
        \nu_\chi(W)=1.
\]

Choose a point \(x_*\in J_{\mathrm{dyn}}\) and define
\[
        g:X_\chi\longrightarrow J_{\mathrm{dyn}}
\]
by
\[
        g(\omega)
        :=
        \begin{cases}
        \kappa_Z^{-1}(\omega),&\omega\in W,\\
        x_*,&\omega\notin W.
        \end{cases}
\]
The map \(g\) is Borel. Moreover,
\[
        \kappa\circ g
        =
        \operatorname{id}_{X_\chi}
        \qquad
        \nu_\chi\text{-almost everywhere}.
\tag{8.5}
\label{eq:BT-kappa-section}
\]

We next verify that \(g\) intertwines the symbolic and interval dynamics
almost everywhere. Since each point has at most three preimages under
\(S\), the set \(S^{-1}(E_0)\) is countable. Therefore
\[
        N
        :=
        (X_\chi\setminus W)
        \cup
        \sigma^{-1}(X_\chi\setminus W)
        \cup
        \kappa\bigl(S^{-1}(E_0)\bigr)
\]
is \(\nu_\chi\)-null. If \(\omega\notin N\) and
\(x=g(\omega)\), then \(x\in Z\), \(Sx\notin E_0\), and hence
\(Sx\in Z\). Using \(\kappa\circ S=\sigma\circ\kappa\), we obtain
\[
        \kappa_Z(Sx)
        =
        \sigma(\kappa_Z(x))
        =
        \sigma\omega.
\]
Because \(\sigma\omega\in W\), application of \(\kappa_Z^{-1}\) gives
\[
        g(\sigma\omega)=S(g(\omega))
        \qquad
        \text{for }\nu_\chi\text{-almost every }\omega\in X_\chi.
\tag{8.6}
\label{eq:BT-section-intertwining}
\]

Define
\[
        \mu:=g_*\nu_\chi.
\]
Then \(\mu\) is an \(S\)-invariant Borel probability on
\(J_{\mathrm{dyn}}\). Indeed, for every Borel set
\(B\subset J_{\mathrm{dyn}}\), equation
\eqref{eq:BT-section-intertwining} and the \(\sigma\)-invariance of
\(\nu_\chi\) give
\[
\begin{aligned}
\mu(S^{-1}B)
&=
\nu_\chi\bigl(\{\omega:S(g(\omega))\in B\}\bigr)\\
&=
\nu_\chi\bigl(\{\omega:g(\sigma\omega)\in B\}\bigr)\\
&=
\nu_\chi\bigl(\sigma^{-1}(g^{-1}(B))\bigr)\\
&=
\nu_\chi(g^{-1}(B))\\
&=
\mu(B).
\end{aligned}
\]
Equation \eqref{eq:BT-kappa-section} also yields
\[
        \kappa_*\mu=\nu_\chi.
\tag{8.7}
\label{eq:BT-mu-symbolic-pushforward}
\]
Thus existence of an invariant probability on the interval system has
now been established explicitly.

Let \(\lambda\) be any \(S\)-invariant probability on
\(J_{\mathrm{dyn}}\). Since \(S\) has no periodic points, the preceding
atomic-measure observation implies that \(\lambda\) is non-atomic. Because
\(E_0\) is countable,
\[
        \lambda(E_0)=0.
\]

The push-forward
\[
        \nu:=\kappa_*\lambda
\]
is \(\sigma\)-invariant:
\[
\begin{aligned}
\nu(\sigma^{-1}A)
&=
\lambda\bigl(\kappa^{-1}(\sigma^{-1}A)\bigr)\\
&=
\lambda\bigl(S^{-1}(\kappa^{-1}(A))\bigr)\\
&=
\lambda(\kappa^{-1}(A))\\
&=
\nu(A)
\end{aligned}
\]
for every Borel \(A\subset X_\chi\). Since \(X_\chi\) is uniquely
ergodic,
\[
        \nu=\nu_\chi.
\]

Now let \(B\subset Z\) be Borel. By the Lusin--Souslin theorem,
\(\kappa(B)\) is Borel. Since \(\kappa\) is injective on \(Z\),
\[
        \kappa^{-1}(\kappa(B))\cap Z=B,
\]
and therefore
\[
        \kappa^{-1}(\kappa(B))\triangle B
        \subset E_0.
\]
Consequently,
\[
\begin{aligned}
\lambda(B)
&=
\lambda\bigl(\kappa^{-1}(\kappa(B))\bigr)\\
&=
(\kappa_*\lambda)(\kappa(B))\\
&=
\nu_\chi(\kappa(B)).
\end{aligned}
\]
Thus \(\nu_\chi\) uniquely determines \(\lambda\) on \(Z\). Since
\(\lambda(E_0)=0\), it uniquely determines \(\lambda\) on all of
\(J_{\mathrm{dyn}}\). The probability \(\mu=g_*\nu_\chi\) constructed
above has \(\kappa_*\mu=\nu_\chi\), so every \(S\)-invariant probability
\(\lambda\) equals \(\mu\). Hence the interval system is uniquely
ergodic. We regard \(\mu\) as a probability on the compact set \(J_{\mathrm{att}}\) by
assigning zero mass to the added endpoint \(1\).

The branch masses are the symbolic letter frequencies:
\[
\mu(I_1)=\alpha,
\qquad
\mu(I_2)=\alpha-\beta,
\qquad
\mu(I_3)=q.
\tag{8.8}
\label{eq:BT-branch-masses}
\]
For the distribution function \(h(x)=\mu([0,x])\), this gives
\[
        h(1-\alpha)=\alpha,
        \qquad
        h(1-\beta)=2\alpha-\beta.
\tag{8.9}
\label{eq:BT-branch-cuts}
\]

\subsubsection{Exact graph-directed structure and branchwise nonsingularity}

Put
\[
        \widehat J_i:=J_{\mathrm{att}}\cap\overline{I_i},
        \qquad i=1,2,3,
\]
and define
\[
        L(x)=\alpha(x-s),
        \qquad
        M(x)=\alpha x+\beta,
        \qquad
        \phi(x)=s+\alpha x.
\]

We now describe the associated directed multigraph explicitly. We use the
standard graph-directed convention that an edge \(e\) from vertex \(i\) to
vertex \(j\) carries a contraction
\[
        f_e:\overline{I_j}\longrightarrow\overline{I_i}.
\]
The edge set consists of
\[
\begin{array}{c|c}
\text{edge}&f_e\\
\hline
1\to2&L,\\
1\to2&M,\\
1\to3&L,\\
2\to1&\phi,\\
3\to2&\phi,\\
3\to3&\phi.
\end{array}
\]
Repeated formulas in this list represent distinct directed edges.

For a triple
\[
        E=(E_1,E_2,E_3)
\]
of nonempty compact subsets of the corresponding branch closures, define
\[
\begin{aligned}
        \mathcal F_1(E)
        &=
        L(E_2)\cup M(E_2)\cup L(E_3),\\
        \mathcal F_2(E)
        &=
        \phi(E_1),\\
        \mathcal F_3(E)
        &=
        \phi(E_2)\cup\phi(E_3),
\end{aligned}
\]
and put
\[
        \mathcal F(E)
        =
        \bigl(
        \mathcal F_1(E),
        \mathcal F_2(E),
        \mathcal F_3(E)
        \bigr).
\]

\begin{lemma}[Exact graph-directed realization and cylinder identification]
\label{lem:BT-exact-graph-directed}
Let
\[
        \mathcal E^{(0)}
        =
        \bigl(
        \overline{I_1},
        \overline{I_2},
        \overline{I_3}
        \bigr)
\]
and define recursively
\[
        \mathcal E^{(n+1)}
        =
        \mathcal F(\mathcal E^{(n)}).
\]
Write
\[
        \mathcal E^{(n)}
        =
        \bigl(
        E_1^{(n)},E_2^{(n)},E_3^{(n)}
        \bigr).
\]

For an admissible directed path
\[
        \gamma=e_1e_2\cdots e_n,
\]
where the terminal vertex of \(e_r\) is the initial vertex of \(e_{r+1}\),
put
\[
        f_\gamma
        :=
        f_{e_1}\circ f_{e_2}\circ\cdots\circ f_{e_n}.
\]
Let \(i(\gamma)\) and \(t(\gamma)\) denote respectively the initial and
terminal vertices of \(\gamma\).

Then, for every \(n\geq1\) and \(i\in\{1,2,3\}\),
\[
        E_i^{(n)}
        =
        \bigcup_{\substack{|\gamma|=n\\ i(\gamma)=i}}
        f_\gamma\bigl(\overline{I_{t(\gamma)}}\bigr).
\tag{8.10}
\label{eq:BT-cylinder-union}
\]
The sets appearing on the right-hand side are precisely the closed
depth-\(n\) stationary inducing cylinders, including all of their prescribed
tower levels.

Moreover,
\[
        E_i^{(n+1)}\subset E_i^{(n)}
\]
and
\[
        \widehat J_i
        =
        \bigcap_{n\geq0}E_i^{(n)}.
\tag{8.11}
\label{eq:BT-cylinder-intersection}
\]
Consequently, the triple
\[
        (\widehat J_1,\widehat J_2,\widehat J_3)
\]
is the unique nonempty compact fixed point of \(\mathcal F\), and the
following are exact equalities of compact sets:
\[
\begin{aligned}
        \widehat J_1
        &=
        L(\widehat J_2)
        \cup
        M(\widehat J_2)
        \cup
        L(\widehat J_3),\\
        \widehat J_2
        &=
        \phi(\widehat J_1),\\
        \widehat J_3
        &=
        \phi(\widehat J_2)
        \cup
        \phi(\widehat J_3).
\end{aligned}
\tag{8.12}
\label{eq:BT-exact-graph-directed}
\]
\end{lemma}

\begin{proof}
The complete proof is given in Appendix~\ref{app:graph-directed-details}.
\end{proof}

The remaining graph-directed measure analysis---including the strong open set
condition, the Hausdorff-dimension calculation, the identification of the
invariant measure with normalized Hausdorff measure, the Bratteli--Vershik
intertwining, and branchwise nonsingularity---is given in
Appendix~\ref{app:graph-directed-details}.  In particular, that appendix proves
the Hausdorff-measure identity \eqref{eq:BT-Hausdorff-measure} and the
branchwise nonsingularity used below.
The same appendix also proves that the compact attractor is exactly the support of the
invariant measure constructed above. From this point onward, in agreement with Sections~2--7,
we therefore write
\[
        J:=\supp\mu=J_{\mathrm{att}},
        \qquad
        J_{\mathrm{dyn}}=J\cap I.
\]

\subsubsection{The complete renormalization gap towers}

The first image of the half-open dynamical interval is
\[
        S(I)=[0,\beta)\cup[\alpha,1),
\]
so the first complementary gap is
\[
        G_0=(\beta,\alpha).
\tag{8.13}
\label{eq:BT-first-gap}
\]
Let
\[
        D_i=\phi(I_i),
        \qquad i=1,2,3.
\]
The first-return tower satisfies
\[
I\setminus G_0
\doteq
D_1\cup D_2\cup D_3
\cup S(D_2)\cup S(D_3)\cup S^2(D_2),
\tag{8.14}
\label{eq:BT-first-tower}
\]
where \(A\doteq B\) means that \(A\triangle B\) is contained in a finite set of tower
endpoints. Indeed,
\[
\begin{aligned}
S(D_2)&=[0,\beta(1-\alpha)),&
S(D_3)&=[\beta(1-\alpha),\beta),\\
S^2(D_2)&=[\alpha,1-\alpha),&
D_1\cup D_2\cup D_3&=[1-\alpha,1).
\end{aligned}
\]

For \(m\geq0\), define
\[
        \Delta_m=\phi^m(I),
        \qquad
        \Delta_{m,i}=\phi^m(I_i),
        \qquad
        \overline{\Delta}_m=\phi^m([0,1]),
        \qquad
        \overline{\Delta}_{m,i}=\phi^m(\overline{I_i}),
\]
and put
\[
        B_m=\phi^m(G_0),
        \qquad
        \ell_m=|\chi^m(1)|.
\]
For \(i\in\{1,2,3\}\), write
\[
        \chi^m(i)
        =
        a_0a_1\cdots a_{r_m(i)-1},
        \qquad
        r_m(i)=|\chi^m(i)|.
\]
For \(0\leq j<r_m(i)\), put
\[
        d_{m,i,j}
        :=
        \sum_{q=0}^{j-1}c_{a_q},
        \qquad
        d_{m,i,0}=0,
\]
and define
\[
        \widetilde S_{m,i,j}(x)
        :=
        x+d_{m,i,j}.
\]
Thus \(\widetilde S_{m,i,j}\) is the affine branch of \(S^j\)
determined by the first \(j\) letters of the return word
\(\chi^m(i)\). It agrees with \(S^j\) on the relative interior of
\(\Delta_{m,i}\). At a tower endpoint it is understood as the
corresponding one-sided affine continuation and need not coincide with
the value selected there by the global half-open definition of \(S^j\).

For \(E\subset\Delta_m\), define the affine level-\(m\) saturation by
\[
        \operatorname{Sat}^{\mathrm{aff}}_m(E)
        :=
        \bigcup_{i=1}^{3}
        \bigcup_{0\leq j<r_m(i)}
        \widetilde S_{m,i,j}
        \bigl(E\cap\Delta_{m,i}\bigr).
\tag{8.15}
\label{eq:BT-affine-saturation}
\]
The half-open geometric tower support is
\[
        U_m
        :=
        \operatorname{Sat}^{\mathrm{aff}}_m(\Delta_m).
\tag{8.16}
\label{eq:BT-tower-support}
\]
It agrees with the corresponding union of the actual \(S\)-orbit
levels away from a finite set of tower endpoints.

For an arbitrary set \(E\subset\overline{\Delta}_m\), define the
closed-level affine saturation by
\[
        \widehat{\operatorname{Sat}}_m(E)
        :=
        \bigcup_{i=1}^{3}
        \bigcup_{0\leq j<r_m(i)}
        \widetilde S_{m,i,j}
        \bigl(E\cap\overline{\Delta}_{m,i}\bigr).
\tag{8.17}
\label{eq:BT-compact-saturation}
\]
When \(E\) is compact, each set in this finite union is compact, and
therefore \(\widehat{\operatorname{Sat}}_m(E)\) is compact. Thus the
notation applies both to compact retained sets such as \(R_m\) and to
arbitrary subsets of \(\overline{\Delta}_m\), without changing the
underlying affine branch convention.
The compact level-\(m\) tower support is
\[
        K_m
        :=
        \widehat{\operatorname{Sat}}_m(\overline{\Delta}_m)
        =
        \bigcup_{i=1}^{3}
        \bigcup_{0\leq j<r_m(i)}
        \widetilde S_{m,i,j}
        \bigl(\overline{\Delta}_{m,i}\bigr).
\tag{8.18}
\label{eq:BT-compact-tower-support}
\]
Since the union is finite and every
\(\widetilde S_{m,i,j}\) is continuous,
\[
\begin{aligned}
        \overline{U_m}^{\,[0,1]}
        &=
        \bigcup_{i=1}^{3}
        \bigcup_{0\leq j<r_m(i)}
        \overline{
        \widetilde S_{m,i,j}(\Delta_{m,i})
        }^{\,[0,1]}\\
        &=
        \bigcup_{i=1}^{3}
        \bigcup_{0\leq j<r_m(i)}
        \widetilde S_{m,i,j}
        \bigl(\overline{\Delta}_{m,i}\bigr)\\
        &=
        K_m.
\end{aligned}
\tag{8.19}
\label{eq:BT-Km-closure-Um}
\]
Here and below, every closure is the ordinary topological closure in
\([0,1]\). The prescribed affine continuations are used only to describe
the closed tower levels; they do not alter the meaning of closure.

\begin{lemma}[Renormalization gap towers]
\label{lem:BT-gap-towers}
For every \(m\geq0\), put
\[
        G_{m,j}=S^j(B_m),
        \qquad
        0\leq j<\ell_m.
\tag{8.20}
\label{eq:BT-gap-towers}
\]
Then the following statements hold.
\begin{enumerate}[label=\textup{(\roman*)}]
\item The first-return itinerary of \(\Delta_{m,i}\) to \(\Delta_m\) is
\(\chi^m(i)\), and its return time is \(r_m(i)\).

\item The retained and removed parts satisfy
\[
        U_{m+1}
        \doteq
        \operatorname{Sat}^{\mathrm{aff}}_m(\Delta_m\setminus B_m)
\]
and
\[
        U_m\setminus U_{m+1}
        \doteq
        \operatorname{Sat}^{\mathrm{aff}}_m(B_m)
        =
        \bigsqcup_{j=0}^{\ell_m-1}G_{m,j}.
\]

\item The compact tower supports satisfy the exact identity
\[
        K_m\setminus K_{m+1}
        =
        \bigsqcup_{j=0}^{\ell_m-1}G_{m,j},
\]
proved below as \eqref{eq:BT-exact-tower-refinement}.

\item The intervals \(G_{m,j}\), over all \(m\geq0\) and
\(0\leq j<\ell_m\), are pairwise disjoint connected components of
\([0,1]\setminus J\), and they exhaust all such components.
\end{enumerate}
\end{lemma}

\begin{proof}
The complete proof is given in Appendix~\ref{app:gap-tower-details}.
\end{proof}

\subsubsection{Classification of the active gaps}

Let
\[
        A_m=G_{m,\ell_m-1}
\tag{8.21}
\label{eq:BT-top-gap}
\]
be the top level of the \(m\)-th gap tower.

\begin{lemma}
\label{lem:BT-active-gaps}
The active gaps are exactly the intervals \(A_m\), \(m\geq0\).
\end{lemma}

\begin{proof}
If \(j<\ell_m-1\), then
\[
        S(G_{m,j})=G_{m,j+1},
\]
which is another complementary gap. Hence \(\mu(S(G_{m,j}))=0\), so
\(G_{m,j}\) is inactive.

We determine the top gap from the total displacement of its return word. Let \(a_m\) be the
last letter of \(w_m=\chi^m(1)\). Since \(\chi(1)\) ends in \(2\), whereas
\(\chi(2)\) and \(\chi(3)\) end in \(1\),
\[
        a_m=
        \begin{cases}
        1,&m\text{ even},\\
        2,&m\text{ odd}.
        \end{cases}
\tag{8.22}
\label{eq:BT-last-letter}
\]
On \(\Delta_{m,1}=\phi^m(I_1)\), stationary inducing gives
\[
        S^{\ell_m}(\phi^m(x))
        =\phi^m(Sx)
        =\phi^m(x+\alpha)
        =\phi^m(x)+\alpha^{m+1}.
\]
Thus the total translation along \(w_m\) is \(\alpha^{m+1}\). Removing the final branch
translation yields
\[
        A_m
        =B_m+\alpha^{m+1}-c_{a_m},
        \qquad
        c_1=\alpha,
        \quad
        c_2=\beta.
\tag{8.23}
\label{eq:BT-top-gap-translation}
\]
Define
\[
        R_1(x)=\alpha+\beta x,
        \qquad
        R_2(x)=\alpha+\beta+\beta x.
\tag{8.24}
\label{eq:BT-R-maps}
\]
Since \(\phi^2(x)=1-\beta+\beta x\), formula
\eqref{eq:BT-top-gap-translation} also gives
\[
        A_{m+2}=
        \begin{cases}
        R_1(A_m),&m\text{ even},\\
        R_2(A_m),&m\text{ odd}.
        \end{cases}
\tag{8.25}
\label{eq:BT-active-recursion}
\]
Here we use
\((1-\beta)(1-\alpha)=\alpha\) and \((1-\beta)^2=\alpha+\beta\), both consequences
of \(P(\alpha)=0\). The initial intervals are
\[
        A_0=(\beta,\alpha),
        \qquad
        A_1=(\alpha+\beta,1-\alpha+\beta).
\tag{8.26}
\label{eq:BT-active-initial}
\]
Since \(R_1\) fixes \(1-\alpha\) and \(R_2\) fixes \(1-\beta\), we obtain
\[
A_{2r}
=
\left(
1-\alpha-\beta^r(1-\alpha-\beta),
\,
1-\alpha-\beta^r(1-2\alpha)
\right),
\tag{8.27}
\label{eq:BT-active-even}
\]
and
\[
A_{2r+1}
=
\left(
1-\beta-\beta^r(1-\alpha-2\beta),
\,
1-\beta-\beta^r(\alpha-2\beta)
\right).
\tag{8.28}
\label{eq:BT-active-odd}
\]

It remains to prove activity. Since \(s\in J\) and \(\phi(J)\subset J\), the points
\[
        z_{2r}=\phi^{2r+1}(s)=1-\beta^{r+1},
        \qquad
        z_{2r+1}=\phi^{2r+2}(s)=1-\alpha\beta^{r+1}
\]
belong to \(J\). For the even family, \(A_{2r}\subset I_1\), and the strict inequalities
placing \(z_{2r}\) inside \(S(A_{2r})=A_{2r}+\alpha\) reduce to
\[
        1-\alpha-2\beta=(1-2\alpha)(1+\alpha)>0,
        \qquad
        \beta-(1-2\alpha)=\alpha^3>0.
\]
For the odd family, \(A_{2r+1}\subset I_2\), and the corresponding inequalities placing
\(z_{2r+1}\) inside \(S(A_{2r+1})=A_{2r+1}+\beta\) reduce to
\[
        1-\alpha-2\beta-\alpha\beta>0,
        \qquad
        \alpha\beta-(\alpha-2\beta)=\alpha^4>0.
\]
For completeness, \(P\) is strictly decreasing on \([0,\frac12]\), while
\(P(0.44)>0>P(0.45)\); hence \(\alpha<0.45\), and
\[
        1-\alpha-2\alpha^2-\alpha^3
        =2-3\alpha-3\alpha^2>0.
\]
Thus each \(S(A_m)\) contains a point of \(J\) in its interior. Since
\(J=\supp\mu\), this open interval has positive \(\mu\)-measure. Therefore every
\(A_m\) is active, and the preceding argument shows that no other gap is active.
\end{proof}

\subsubsection{Exact distribution-coordinate cut values}

The similarities in \eqref{eq:BT-R-maps} arise directly from the graph maps:
\[
        R_1=M\circ\phi,
        \qquad
        R_2=\phi\circ M.
\tag{8.29}
\label{eq:BT-R-compositions}
\]
Indeed, \(\alpha s+\beta=\alpha\) and \(s+\alpha\beta=\alpha+\beta\). We now prove the
two branch-tail identities used in the cut computation. In the decomposition of
\(\widehat J_1\) in \eqref{eq:BT-exact-graph-directed}, the pieces
\(L(\widehat J_2)\cup L(\widehat J_3)\) lie in \([0,\beta]\), while
\[
        M(\widehat J_2)\subset[\alpha,s].
\]
Since \(\beta<\alpha\),
\[
        \widehat J_1\cap[\alpha,s]
        =M(\widehat J_2)
        =M(\phi(\widehat J_1))
        =R_1(\widehat J_1).
\tag{8.30}
\label{eq:BT-right-J1}
\]
Moreover, \(\widehat J_2=\phi(\widehat J_1)\) and
\(\phi([\alpha,s])=[s+\beta,1-\beta]\). Since \(\phi\) is increasing,
\[
\begin{aligned}
\widehat J_2\cap[s+\beta,1-\beta]
&=\phi(\widehat J_1\cap[\alpha,s])\\
&=\phi(R_1(\widehat J_1))\\
&=R_2(\phi(\widehat J_1))
=R_2(\widehat J_2).
\end{aligned}
\tag{8.31}
\label{eq:BT-right-J2}
\]

Let \(a_r\) be the right endpoint of \(A_{2r}=R_1^r(G_0)\). Since \(R_1\) is increasing
and fixes \(s\), induction from \eqref{eq:BT-right-J1} gives
\[
        \widehat J_1\cap[a_r,s]
        =R_1^{r+1}(\widehat J_1).
\tag{8.32}
\label{eq:BT-even-tail}
\]
Similarly, if \(b_r\) is the right endpoint of
\(A_{2r+1}=R_2^r(A_1)\), then \(R_2\) fixes \(1-\beta\), and
\eqref{eq:BT-right-J2} gives
\[
        \widehat J_2\cap[b_r,1-\beta]
        =R_2^{r+1}(\widehat J_2).
\tag{8.33}
\label{eq:BT-odd-tail}
\]

Both \(R_1\) and \(R_2\) have contraction ratio \(\beta=\alpha^2\). By
\eqref{eq:BT-Hausdorff-measure}, for every Borel subset \(E\) of the relevant compact branch
piece,
\[
        \mu(R_i(E))
        =\beta^d\mu(E)
        =(\alpha^d)^2\mu(E)
        =s^2\mu(E)
        =q\mu(E),
        \qquad i=1,2.
\tag{8.34}
\label{eq:BT-measure-scaling}
\]
Because \(h\) is constant across a gap, we may evaluate it at the right endpoint. Using
\eqref{eq:BT-even-tail} and \(\mu(\widehat J_1)=\alpha\),
\[
\begin{aligned}
h(A_{2r})
&=\alpha-\mu(R_1^{r+1}(\widehat J_1))\\
&=\alpha(1-q^{r+1}).
\end{aligned}
\tag{8.35}
\label{eq:BT-even-cuts}
\]
Similarly, using \eqref{eq:BT-odd-tail},
\(h(1-\beta)=2\alpha-\beta\), and
\(\mu(\widehat J_2)=\alpha-\beta\),
\[
        h(A_{2r+1})
        =2\alpha-\beta-(\alpha-\beta)q^{r+1}.
\tag{8.36}
\label{eq:BT-odd-cuts}
\]
Put
\[
        u_r=\alpha(1-q^{r+1}),
        \qquad
        v_r=2\alpha-\beta-(\alpha-\beta)q^{r+1}.
\tag{8.37}
\label{eq:BT-cut-sequences}
\]
Then
\[
        u_r\uparrow\alpha,
        \qquad
        v_r\uparrow2\alpha-\beta.
\tag{8.38}
\label{eq:BT-cut-limits}
\]

\begin{theorem}[Explicit canonical defect cut set for a self-similar infinite-type ITM]
\label{thm:BT-explicit-Cdef}
For the infinite-type map \eqref{eq:BT-map} and its unique invariant probability measure,
\[
\boxed{
\begin{aligned}
\cC_{\mathrm{def}}
={}&
\{0,1,\alpha,2\alpha-\beta\}\\
&\cup\{u_r:r\geq0\}\\
&\cup\{v_r:r\geq0\}.
\end{aligned}
}
\tag{8.39}
\label{eq:BT-Cdef}
\]
In particular, \(\cC_{\mathrm{def}}\) is closed, countably infinite and Lebesgue-null. Its
only nontrivial accumulation points are
\[
        \alpha
        \qquad\text{and}\qquad
        2\alpha-\beta.
\]
Consequently, the ordinary distribution-coordinate model is a countable IET.
\end{theorem}

\begin{proof}
By the branchwise-nonsingularity argument above, the invariant measure is branchwise
nonsingular. Hence
Proposition~\ref{prop:defect-support-active-closure} applies. By
Lemma~\ref{lem:BT-active-gaps}, the active gaps are exactly \(A_m\). The even family lies in
\(I_1\), the odd family lies in \(I_2\), and there are no active gaps in \(I_3\). Equations
\eqref{eq:BT-even-cuts}--\eqref{eq:BT-odd-cuts} give
\[
\supp\rho_1=\{u_r:r\geq0\}\cup\{\alpha\},
\]
\[
\supp\rho_2=\{v_r:r\geq0\}\cup\{2\alpha-\beta\},
\qquad
\rho_3=0.
\]
The branch-cut values are
\[
        0,
        \quad
        \alpha,
        \quad
        2\alpha-\beta,
        \quad
        1.
\]
Here the final value \(1=h(1)\) belongs to the compact distribution-coordinate interval;
it corresponds to the null endpoint added when \(I=[0,1)\) is compactified to \([0,1]\).
This proves \eqref{eq:BT-Cdef}. The remaining assertions follow from
\eqref{eq:BT-cut-limits} and Theorem~\ref{thm:defect-set-structure}\textup{(iii)}.
\end{proof}

\begin{lemma}[Active-gap cuts are essential]\label{lem:active-gap-essential-cut}
Let \(G=(a,b)\) be an active complementary gap such that
\[
        [a,b]\subset I_k^\circ,
\]
and suppose that
\[
        c:=h(a)=h(b)\in(0,1).
\]
Then the displacement
\[
        D(u)=T(u)-u
\]
has one-sided limits along the full-measure coordinate set \(Y\), and
\[
        D(c+)-D(c-)=\mu(\tau_k(G))>0.
\]
Consequently, no neighborhood of \(c\) can carry a single translation representative of
\(T\) almost everywhere.
\end{lemma}

\begin{proof}
Because \(G\) is a full complementary component of \(J\) and \(\mu\) is non-atomic,
\[
        h^{-1}(\{c\})=[a,b].
\]
Indeed, if the fiber were larger, it would contain a larger open interval disjoint from the
support, contradicting maximality of the complementary component \(G\).

Since \(0<c<1\) and \(Y\) has full Lebesgue measure, \(Y\) meets every sufficiently small
interval on both sides of \(c\). Because \(q:Y\to X\) is the increasing inverse
distribution coordinate, continuity and monotonicity of \(h\) give
\[
        q(u)\longrightarrow a
        \quad(u\uparrow c,\ u\in Y),
\]
and
\[
        q(u)\longrightarrow b
        \quad(u\downarrow c,\ u\in Y).
\]
For \(u\in Y\) sufficiently close to \(c\), the point \(x=q(u)\) lies in \(I_k\), and
\[
        D(u)=h(x+c_k)-h(x)=\Delta_k(x).
\]
The function \(\Delta_k\) is continuous on a neighborhood of \([a,b]\), because
\([a,b]\subset I_k^\circ\). Hence
\[
        D(c-)=\Delta_k(a),
        \qquad
        D(c+)=\Delta_k(b).
\]
By Lemma~\ref{lem:defect-measure-identity} and the fact that \(\mu(G)=0\),
\[
\begin{aligned}
D(c+)-D(c-)
&=
\Delta_k(b)-\Delta_k(a)\\
&=
\sigma_k(G)\\
&=
\mu(\tau_k(G))-\mu(G)\\
&=
\mu(\tau_k(G))
>0.
\end{aligned}
\]

Suppose that \(T(u)=u+d\) almost everywhere on a neighborhood \(V\) of \(c\). The set on
which this equality holds has full measure in \(V\), and therefore contains sequences
approaching \(c\) from both sides through \(Y\). Along both sequences \(D(u)=d\), forcing
\[
        D(c-)=D(c+)=d,
\]
a contradiction.
\end{proof}

\begin{proposition}[The canonical model is genuinely non-finite]
\label{prop:BT-nonfinite-canonical}
The countable IET obtained in Theorem~\ref{thm:BT-explicit-Cdef} cannot be
represented by finitely many translation intervals in the ordinary distribution coordinate.
\end{proposition}

\begin{proof}
By Lemma~\ref{lem:BT-active-gaps}, every \(A_m\) is an active full gap lying in the interior
of one branch. The explicit formulas give
\[
        0<u_r<\alpha<1
\]
and
\[
        \alpha<v_r<2\alpha-\beta<1.
\]
Thus every active-gap cut lies in \((0,1)\), and
Lemma~\ref{lem:active-gap-essential-cut} shows that each
\[
        h(A_{2r})=u_r,
        \qquad
        h(A_{2r+1})=v_r
\]
is an essential change of translation. Since \(0<q<1\), each of the two sequences is
strictly increasing. Moreover,
\[
        u_r<\alpha<v_s
        \qquad(r,s\geq0),
\]
because
\[
        v_s-\alpha=(\alpha-\beta)(1-q^{s+1})>0.
\]
Thus these essential cuts are pairwise distinct.

If the canonical map admitted a finite IET representative, it would be a translation on each
component of the complement of a finite cut set. Every essential cut above would have to
belong to that finite set, which is impossible. Hence the canonical countable IET cannot be
represented by finitely many translation intervals.
\end{proof}

\begin{remark}
The Arnoux--Ornstein--Weiss theorem already supplies an abstract countable IET model. The content of Theorem~\ref{thm:BT-explicit-Cdef} is different: the ordinary
order-preserving distribution coordinate itself works, its canonical defect cut set is calculated
explicitly from the renormalization gap towers, and the resulting canonical model is proved
to require infinitely many pieces.
\end{remark}

\begin{proposition}[Full support]\label{prop:full-support}
If \(J=[0,1]\), then \(\cC_{\mathrm{gap}}=\emptyset\). If, in addition, \(\mu\) is
branchwise nonsingular on the support, then \((J,\mu,S)\) is metrically equivalent, through
the distribution coordinate, to a finite IET.
\end{proposition}

\begin{proof}
There are no gaps, so the active-gap cut set is empty. The conclusion follows from
Theorem~\ref{thm:finite-active-gap}.
\end{proof}

\begin{proposition}[Finite union of intervals]\label{prop:finite-union}
If \(J\) is a finite union of closed intervals, then \(\cC_{\mathrm{gap}}\) is finite. If,
in addition, \(\mu\) is branchwise nonsingular on the support, then
\((J,\mu,S)\) is metrically equivalent to a finite IET.
\end{proposition}

\begin{proof}
The complement of \(J\) has finitely many connected components. Intersecting them with the
finitely many continuity branches produces only finitely many branchwise gaps, so only
finitely many active-gap cuts can occur. Apply Theorem~\ref{thm:finite-active-gap}.
\end{proof}

\begin{proposition}[Finite-type Lebesgue model]\label{prop:finite-type-lebesgue}
Let \(A\subset I=[0,1)\) be a finite union of intervals with
\[
        \Leb(A)>0.
\]
Assume that \(S(A)=A\) modulo finitely many endpoints and that \(S|_A\) is one-to-one
modulo endpoints. Define
\[
        \mu_A=\frac{\Leb|_A}{\Leb(A)}.
\]
Then \(\mu_A\) is \(S\)-invariant and branchwise nonsingular. Its support is the closure in
\([0,1]\) of the union of the positive-length interval components of \(A\), and is therefore
a finite union of closed intervals. Consequently, the distribution-coordinate model of
\((\supp\mu_A,\mu_A,S)\) is a finite IET.
\end{proposition}

\begin{proof}
After merging overlapping components, let \(A_+\) be the union of the positive-length
interval components of \(A\), and put
\[
        K_A:=\overline{A_+}^{\,[0,1]}.
\]
The set \(K_A\) is a finite union of closed intervals.

For each continuity branch \(I_k\), put \(A_k=A\cap I_k\). On \(A_k\), the map \(S\) is
the translation \(\tau_k\), and therefore preserves Lebesgue measure. The one-to-one
assumption implies that the sets \(S(A_k)\) are pairwise disjoint modulo null endpoint sets,
while \(S(A)=A\) modulo endpoints says that their union agrees with \(A\) modulo a null set.
Thus, for every Borel set \(B\subset[0,1]\),
\[
\begin{aligned}
\Leb(S^{-1}(B)\cap A)
&=
\sum_k\Leb(S^{-1}(B)\cap A_k)\\
&=
\sum_k\Leb(B\cap S(A_k))\\
&=
\Leb(B\cap A).
\end{aligned}
\]
Division by \(\Leb(A)>0\) gives
\[
        \mu_A(S^{-1}(B))=\mu_A(B),
\]
so \(\mu_A\) is invariant.

We next identify the support. If \(x\in K_A\), every neighborhood of \(x\) meets a
positive-length component of \(A\) in a set of positive Lebesgue measure; hence every
neighborhood has positive \(\mu_A\)-measure. If \(x\notin K_A\), some neighborhood of
\(x\) is disjoint from \(A_+\), and its intersection with \(A\) consists at most of finitely
many endpoints or degenerate components, so it has zero Lebesgue measure. Therefore
\[
        \supp\mu_A=K_A.
\]

Finally, let \(N\subset K_A\cap I_k\) be Borel and suppose that \(\mu_A(N)=0\). Then
\(\Leb(N)=0\), and translations preserve Lebesgue-null sets, so
\[
        \Leb(\tau_k(N))=0.
\]
Consequently,
\[
        \mu_A(\tau_k(N))
        =
        \frac{\Leb(\tau_k(N)\cap A)}{\Leb(A)}
        =0.
\]
Thus \(\mu_A\) is branchwise nonsingular. Since \(K_A\) is a finite union of closed
intervals, Proposition~\ref{prop:finite-union} gives a finite IET model.
\end{proof}

\begin{remark}[Equivalent descriptions of active gaps]\label{rem:active-gap-reformulations}
For a branchwise gap \(G\subset I_k^\circ\setminus J\), the following are equivalent:
\[
\begin{aligned}
&G\text{ is active},\\
&\mu(G+c_k)>0,\\
&(G+c_k)\cap J\neq\emptyset,\\
&G\cap(J-c_k)\neq\emptyset,\\
&G\cap\bigl((I_k^\circ\setminus J)\cap\tau_k^{-1}(J)\bigr)\neq\emptyset.
\end{aligned}
\]
The equivalence between the second and third conditions uses that \(G+c_k\) is open and
\(J=\supp\mu\). Consequently, finiteness of the active-gap set may be checked either by
counting gaps crossed by the shifted support \(J-c_k\), or by counting components met by the
external preimage set
\[
        (I_k^\circ\setminus J)\cap\tau_k^{-1}(J).
\]
Thus the several geometric finiteness conditions that one can formulate in these terms are
reformulations of the same active-gap criterion rather than independent hypotheses.
\end{remark}

\begin{remark}[When branchwise nonsingularity is automatic]\label{rem:nonsingularity-automatic}
For a fixed branch \(I_k\), define a measure \(\lambda_k\) on \(J\cap I_k\) by
\[
        \lambda_k(A)=\mu(\tau_k(A)).
\]
Branchwise nonsingularity on this branch is exactly
\[
        \lambda_k\ll\mu|_{J\cap I_k}.
\]
A useful sufficient condition is that Lebesgue measure be absolutely continuous with respect
to \(\mu\) on the branch domain, while \(\mu\) is absolutely continuous with respect to
Lebesgue measure on the translated region. In particular, the condition holds when \(\mu\)
is equivalent to Lebesgue measure on the relevant domain and image intervals. It also holds
for the normalized Lebesgue measure in Proposition~\ref{prop:finite-type-lebesgue}. No such
conclusion is asserted for an arbitrary invariant measure of a finite-type map. For singular
Cantor-type supports, branchwise nonsingularity remains a genuine condition.
\end{remark}

\subsection{Why direct-image invariance is not automatic}
\label{subsec:direct-image-failure}

We conclude this section with a simple example illustrating why almost-everywhere
invertibility does not by itself justify the ordinary direct-image argument in the
distribution coordinate. The example also shows concretely how an active gap is recorded by
the branch defect measure.

\begin{example}[Almost-everywhere invertibility does not imply direct-image invariance]
\label{ex:direct-image-failure}
Define
\[
S(x)=
\begin{cases}
 x, & 0\le x<\dfrac12,\\[2mm]
 x-\dfrac12, & \dfrac12\le x<1,
\end{cases}
\qquad
\mu=2\,\Leb|_{[0,1/2]}.
\]
Then $\mu$ is a non-atomic $S$-invariant probability measure and
\[
J=\supp\mu=\left[0,\frac12\right].
\]
Indeed, for every Borel set $A\subset[0,1]$,
\[
\begin{aligned}
\mu(S^{-1}A)
&=2\Leb\left(S^{-1}A\cap\left[0,\frac12\right]\right)\\
&=2\Leb\left(A\cap\left[0,\frac12\right]\right)
=\mu(A),
\end{aligned}
\]
because $S$ is the identity on the support of $\mu$. In particular, $S$ is invertible on the
full-measure invariant set $[0,\frac12)$.

Now take
\[
B=\left[\frac12,\frac34\right)
\]
in the second branch. Since $B$ lies outside the support,
\[
\mu(B)=0.
\]
However,
\[
S(B)=\left[0,\frac14\right),
\qquad
\mu(S(B))=\frac12.
\]
Thus
\[
\mu(S(B))\ne\mu(B),
\]
even though $S$ is invertible $\mu$-almost everywhere.

There is no contradiction with invariance. In fact,
\[
S^{-1}(S(B))
=
\left[0,\frac14\right)
\cup
\left[\frac12,\frac34\right),
\]
and therefore
\[
\mu(S^{-1}(S(B)))=\frac12=\mu(S(B)).
\]
The problem is that the full preimage of $S(B)$ contains a second branch of positive measure.

In the language of Section~6, the interval $G=(\frac12,1)$ is an active branchwise gap,
since
\[
\mu(S(G))=1.
\]
The distribution function is
\[
h(x)=
\begin{cases}
2x, & 0\le x\le\frac12,\\
1, & \frac12\le x\le1,
\end{cases}
\]
and the entire active gap is collapsed to the cut value $h(G)=1$. Moreover,
\[
\rho_2=\delta_1,
\qquad
\cC_{\mathrm{def}}=\{0,1\}.
\]
This example shows concretely why Corollary~\ref{cor:direct-core} applies to sets inside
the invertible core, but not to arbitrary geometric intervals in the original domain.
\end{example}

Let $a<b$ lie in one branch $I_k$. A naive argument would try to write
\[
        h(Sb)-h(Sa)=\mu([Sa,Sb])=\mu(S([a,b]))=\mu([a,b])=h(b)-h(a).
\]
The equality
\[
        \mu(S([a,b]))=\mu([a,b])
\]
is precisely the delicate one. Ordinary invariance gives equality for full preimages, but it does not by itself give equality for direct images of geometric intervals. Such an interval may contain relatively open gaps of the support, or more subtle null subsets inside the support, whose branch images carry positive measure.

The criteria above repair this step in two related ways. The active-gap criterion cuts at the collapsed values of all active gaps and assumes branchwise nonsingularity to rule out hidden singular defects. The defect-measure criterion records the whole failure of direct-image preservation in the positive measures $\sigma_k$; finiteness of the corresponding defect support is a canonical sufficient condition ensuring that the branch displacement is constant on finitely many pieces.

Thus the finite-IET conclusion is not obtained by ignoring the bad sets. It is obtained by locating the places where the distribution coordinate can change the branch displacement and by cutting at those places.

\section{Open questions}

The defect-measure formulation suggests several concrete problems.

\begin{question}
For which classes of interval translation maps and invariant non-atomic measures is the defect
cut set \(\cC_{\mathrm{def}}\) finite?
\end{question}

\begin{question}
For which classes of interval translation maps and invariant non-atomic measures is
branchwise nonsingularity on the support automatic?
\end{question}

\begin{question}
Suppose that the invariant measure is branchwise nonsingular on its support. Can the closure
of the active-gap cut set satisfy
\[
        \Leb\bigl(\overline{\cC_{\mathrm{gap}}}\bigr)>0?
\]
More strongly, can \(\overline{\cC_{\mathrm{gap}}}\) have nonempty interior?
\end{question}

\begin{question}
Suppose that
\[
        \Leb(\cC_{\mathrm{def}})>0.
\]
Can the map induced by the ordinary distribution coordinate,
\[
        T(u)=h\bigl(S(q(u))\bigr),
\]
nevertheless admit a countable-IET representation modulo Lebesgue-null sets?

More generally, can there exist a closed Lebesgue-null set
\[
        C\subsetneq\cC_{\mathrm{def}}
\]
such that the displacement
\[
        D(u)=T(u)-u
\]
is locally constant almost everywhere on every connected component of
\([0,1]\setminus C\)?
\end{question}

The second question concerns the canonical order-preserving distribution coordinate. It is
distinct from the abstract representation problem, since Corollary~\ref{cor:universal-countable-iem}
already gives a countable-IET model through a possibly non-order-preserving measurable
conjugacy. The set \(\cC_{\mathrm{def}}\) is a canonical sufficient cut set, but it is not
claimed to be minimal.

\begin{question}
In a recent preprint, Drach, Staresinic and van Strien prove that finite-type interval
translation maps contain an open dense subset of every finite-branch parameter space
\cite{DrachStaresinicVanStrien2026}. For natural invariant measures arising in other
infinite-type Cantor-attractor families, is \(\Leb(\cC_{\mathrm{def}})=0\) typical? Under
what renormalization hypotheses can the canonical defect cut set be described explicitly, as in
Theorem~\ref{thm:BT-explicit-Cdef}?
\end{question}

The last question concerns the canonical distribution coordinate for a prescribed invariant
measure, and is therefore different from the already established open-density of finite-type
maps. The weakly mixing and uniquely ergodic infinite-type classes studied in
\cite{BruinRadinger2026,SkripchenkoChernyi2026} are natural candidates for further
investigation.

\appendix

\section{Entropy, measurable branch recovery, and the automorphic core}
\label{app:entropy-automorphic-details}

This appendix contains the proofs of the auxiliary entropy, branch-recovery, and
automorphic-core lemmas used in Section~\ref{sec:ae-invertibility}. The proof of
Theorem~\ref{thm:ae-invertible} itself is kept in the main text so that the central logical
passage from polynomial coding complexity to almost-everywhere invertibility remains visible.

\subsection{Lemma~\ref{lem:entropy-identity}}\label{app:lem-entropy-identity}

\begin{proof}[Proof of Lemma~\ref{lem:entropy-identity}]
For \(r\geq0\), let
\[
        \mathcal F_r
        =
        \sigma\left(\bigvee_{j=1}^{r}S^{-j}\cP\right),
        \qquad
        a_r=H_\mu(\cP\mid\mathcal F_r),
\]
with \(\mathcal F_0\) the trivial sigma-algebra, so that \(a_0=H_\mu(\cP)\). The
sequence \((\mathcal F_r)\) is increasing and therefore \((a_r)\) is decreasing. Put
\[
        \mathcal F_\infty
        =
        \sigma\left(\bigvee_{j=1}^{\infty}S^{-j}\cP\right).
\]
Continuity of conditional entropy for a finite partition along increasing sigma-algebras
gives
\[
        a_r\longrightarrow
        a_\infty:=H_\mu(\cP\mid\mathcal F_\infty);
\]
see, for example, \cite[Section~4.3]{Walters}. This continuity is a standard consequence of
martingale convergence for the conditional probabilities of the atoms of \(\cP\).

By the chain rule for entropy,
\[
H_\mu\left(\bigvee_{j=0}^{m-1}S^{-j}\cP\right)
=
\sum_{r=0}^{m-1}
H_\mu\left(S^{-r}\cP\,\middle|\,\bigvee_{j=r+1}^{m-1}S^{-j}\cP\right).
\]
For the \(r\)-th summand, factor out the common pullback \(S^{-r}\). Since \(S\)
preserves \(\mu\), a common pullback preserves the joint distributions of the partitions,
and therefore
\[
H_\mu\left(S^{-r}\cP\,\middle|\,\bigvee_{j=r+1}^{m-1}S^{-j}\cP\right)
=
H_\mu\left(\cP\,\middle|\,\bigvee_{j=1}^{m-1-r}S^{-j}\cP\right)
=a_{m-1-r}.
\]
Consequently,
\[
        H_\mu(\cP^{(m)})=\sum_{r=0}^{m-1}a_r.
\]
After division by \(m\), the right-hand side is the Cesaro average of the convergent
sequence \((a_r)\), and hence tends to \(a_\infty\). This proves the identity.
\end{proof}

\subsection{Lemma~\ref{lem:branch-future}}\label{app:lem-branch-future}

\begin{proof}[Proof of Lemma~\ref{lem:branch-future}]
By Lemma~\ref{lem:entropy-identity} and the hypothesis,
\[
        H_\mu(\cP\mid\cP^+)=0.
\]
For a finite measurable partition, vanishing conditional entropy is equivalent to
measurability of the partition with respect to the completion of the conditioning
sigma-algebra; see, for example, \cite[Section~4.3]{Walters}. Hence each atom of \(\cP\)
agrees modulo \(\mu\) with a \(\cP^+\)-measurable set. Therefore the branch label
\(\alpha\) is \(\cP^+\)-measurable modulo \(\mu\).
\end{proof}

\subsection{Lemma~\ref{lem:branch-image}}\label{app:lem-branch-image}

\begin{proof}[Proof of Lemma~\ref{lem:branch-image}]
Let
\[
        \mathcal F=\sigma\left(\bigvee_{j=0}^{\infty}S^{-j}\cP\right).
\]
Then
\[
        S^{-1}\mathcal F=\sigma\left(\bigvee_{j=1}^{\infty}S^{-j}\cP\right)=\cP^+.
\]
By Lemma~\ref{lem:branch-future}, each atom $I_k$ of $\cP$ agrees modulo $\mu$ with a set in $S^{-1}\mathcal F$. Hence for every $k$ there exists $B_k\in\mathcal F$ such that
\[
        I_k=S^{-1}B_k \quad \text{mod }\mu.
\]
The sets $B_k$ are pairwise disjoint modulo $\mu$. Indeed, for $k\neq\ell$,
\[
        \mu(B_k\cap B_\ell)
        =\mu(S^{-1}(B_k\cap B_\ell))=0,
\]
because $S^{-1}B_k$ and $S^{-1}B_\ell$ agree modulo null sets with the disjoint
branches $I_k$ and $I_\ell$. Similarly,
\[
        \mu\left(I\setminus\bigcup_{k=1}^nB_k\right)
        =\mu\left(S^{-1}\left(I\setminus\bigcup_{k=1}^nB_k\right)\right)=0,
\]
since the branches cover $I$. After modifying the sets on these null overlaps and the
null uncovered part, we may therefore assume that the $B_k$ form a measurable partition.
Define
\[
        \beta(y)=k\quad\Longleftrightarrow\quad y\in B_k.
\]
Then, for $\mu$-almost every $x$,
\[
        x\in I_k\quad\Longleftrightarrow\quad Sx\in B_k,
\]
which is exactly $\alpha(x)=\beta(Sx)$.
\end{proof}

\subsection{Lemma~\ref{lem:automorphic-core}}\label{app:lem-automorphic-core}

\begin{proof}[Proof of Lemma~\ref{lem:automorphic-core}]
Let
\[
        G=\{z\in Z:R(Sz)=z\}\cap\{z\in Z:S(Rz)=z\}.
\]
The set \(G\) is Borel: for a standard Borel space the diagonal in \(Z\times Z\) is
Borel, and each equality set above is the inverse image of the diagonal under a measurable
map. By hypothesis, \(\mu(G)=1\).

First note that \(R\) also preserves \(\mu\). For every Borel set \(B\subset Z\),
using the \(S\)-invariance of \(\mu\) and the identity \(R\circ S=\mathrm{id}\) almost
everywhere, we have
\[
        \mu(R^{-1}B)
        =\mu(S^{-1}(R^{-1}B))
        =\mu((R\circ S)^{-1}B)
        =\mu(B).
\]
Now put
\[
        Z_0=\bigcap_{m\geq0}S^{-m}G\cap\bigcap_{m\geq0}R^{-m}G.
\]
Since both \(S\) and \(R\) preserve \(\mu\), the set \(Z_0\) is Borel and has full
measure.

Take \(z\in Z_0\). All nonnegative iterates of \(z\) under both maps lie in \(G\). In
particular, \(R(Sz)=z\) and \(S(Rz)=z\). For \(m\geq1\),
\[
        R^m(Sz)=R^{m-1}z,
        \qquad
        S^m(Rz)=S^{m-1}z,
\]
while the remaining iterates are obtained by shifting the corresponding orbit by one step.
Hence \(Sz\in Z_0\) and \(Rz\in Z_0\), so both maps send \(Z_0\) into itself. Since
\(S(Rz)=z\) and \(R(Sz)=z\) on \(Z_0\), these inclusions are equalities and the two
restrictions are inverse bijections. The restriction of \(S\) is measure preserving because
\(S\) preserves \(\mu\) and \(Z_0\) is invariant.
\end{proof}

\section{Graph-directed cylinders, Hausdorff measure, and Bratteli--Vershik invariance}
\label{app:graph-directed-details}

This appendix contains the full graph-directed and measure-theoretic details
deferred from Section~8.1.2.

\subsection{Proof of the exact graph-directed realization}

\begin{proof}[Proof of Lemma~\ref{lem:BT-exact-graph-directed}]
We first identify the graph edges with the levels of the first-return
towers.

Recall that
\[
        D_i=\phi(I_i),
        \qquad i=1,2,3,
\]
and that the return words are
\[
        1\mapsto2,
        \qquad
        2\mapsto311,
        \qquad
        3\mapsto31.
\]
All calculations below are made on branch interiors and then extended to
the corresponding closed intervals by the prescribed continuous affine
extensions.

For \(x\in I_1\), the point \(\phi(x)\) lies in \(I_2\), and the return
word has length one. Hence the closed tower over \(D_1\) consists of the
single level
\[
        \phi(\overline{I_1})
        \subset\overline{I_2}.
\]
This is the edge \(2\to1\) carrying the map \(\phi\).

For \(x\in I_2\), the point \(\phi(x)\) lies in \(I_3\). Since
\[
\begin{aligned}
        S(\phi(x))
        &=
        \phi(x)+\beta-1\\
        &=
        s+\alpha x+\beta-1\\
        &=
        \alpha(x-s)\\
        &=
        L(x),
\end{aligned}
\]
and \(L(x)\in I_1\), we also have
\[
\begin{aligned}
        S^2(\phi(x))
        &=
        L(x)+\alpha\\
        &=
        \alpha x+\beta\\
        &=
        M(x).
\end{aligned}
\]
Thus the three closed levels of the return tower over \(D_2\) are
\[
        \phi(\overline{I_2})\subset\overline{I_3},
        \qquad
        L(\overline{I_2})\subset\overline{I_1},
        \qquad
        M(\overline{I_2})\subset\overline{I_1}.
\]
These are respectively the edges
\[
        3\to2,\qquad 1\to2,\qquad 1\to2.
\]

For \(x\in I_3\), the point \(\phi(x)\) lies in \(I_3\), and the same
calculation gives
\[
        S(\phi(x))=L(x)\in I_1.
\]
Thus the two closed levels of the return tower over \(D_3\) are
\[
        \phi(\overline{I_3})\subset\overline{I_3},
        \qquad
        L(\overline{I_3})\subset\overline{I_1}.
\]
These are the edges \(3\to3\) and \(1\to3\).

The precise ranges are
\[
\begin{aligned}
        L(\overline{I_2})
        &=[0,\beta s],\\
        L(\overline{I_3})
        &=[\beta s,\beta],\\
        M(\overline{I_2})
        &=[\alpha,s],\\
        \phi(\overline{I_1})
        &=[s,1-\beta],\\
        \phi(\overline{I_2})
        &=[1-\beta,1-\alpha\beta],\\
        \phi(\overline{I_3})
        &=[1-\alpha\beta,1].
\end{aligned}
\tag{B.1}
\label{eq:BT-first-level-ranges}
\]
In particular,
\[
\begin{aligned}
        E_1^{(1)}
        &=
        [0,\beta]\cup[\alpha,s],\\
        E_2^{(1)}
        &=
        [s,1-\beta],\\
        E_3^{(1)}
        &=
        [1-\beta,1].
\end{aligned}
\]
Thus
\[
        E_i^{(1)}\subset E_i^{(0)}.
\tag{B.2}
\label{eq:BT-first-nesting}
\]

We next prove the cylinder formula
\eqref{eq:BT-cylinder-union}. For \(n=1\), it is exactly the first-return
tower calculation above.

Assume that it holds at depth \(n\). By the definition of
\(\mathcal F\),
\[
        E_i^{(n+1)}
        =
        \bigcup_{\substack{e\\i(e)=i}}
        f_e\bigl(E_{t(e)}^{(n)}\bigr).
\]
Using the induction hypothesis,
\[
\begin{aligned}
        E_i^{(n+1)}
        &=
        \bigcup_{\substack{e\\i(e)=i}}
        \;
        \bigcup_{\substack{|\gamma|=n\\
        i(\gamma)=t(e)}}
        f_e\circ f_\gamma
        \bigl(\overline{I_{t(\gamma)}}\bigr)\\
        &=
        \bigcup_{\substack{|\eta|=n+1\\i(\eta)=i}}
        f_\eta
        \bigl(\overline{I_{t(\eta)}}\bigr).
\end{aligned}
\]
This proves \eqref{eq:BT-cylinder-union} for every \(n\).

It remains to justify that these graph cylinders are exactly the stationary
inducing cylinders of the original interval map. At depth one this was
proved by the explicit tower calculation above. Because \((\alpha,\beta)\)
is a fixed point of the renormalization, the first-return map on every
renormalized base
\[
        \Delta_m=\phi^m(I)
\]
is, after rescaling by \(\phi^m\), the same three-branch map \(S\), with
the same return words
\[
        2,\qquad311,\qquad31.
\]
Consequently, refining any closed depth-\(n\) inducing cylinder consists
of applying exactly one of the six first-level edge maps listed above in
the corresponding rescaled coordinate. Appending an admissible graph edge
therefore gives exactly one closed depth-\((n+1)\) stationary inducing
cylinder. Conversely, every depth-\((n+1)\) cylinder is obtained in this
way from its unique depth-\(n\) parent, except that two adjacent cylinders
may share a common endpoint. Thus the graph cylinders in
\eqref{eq:BT-cylinder-union} are precisely the closed stationary inducing
cylinders.

Since \(\mathcal F\) is monotone with respect to inclusion,
\eqref{eq:BT-first-nesting} implies inductively that
\[
        E_i^{(n+1)}\subset E_i^{(n)}
        \qquad
        (n\geq0).
\]
Hence
\[
        K_i:=\bigcap_{n\geq0}E_i^{(n)}
\]
is a nonempty compact set.

We first record the cylinder-tree construction used for the reverse
inclusion. Let \(x\in K_i\). For every \(n\), let
\(\mathscr T_n(x)\) be the finite collection of admissible paths
\(\gamma\) of length \(n\), beginning at vertex \(i\), such that
\[
        x\in
        f_\gamma
        \bigl(\overline{I_{t(\gamma)}}\bigr).
\]
Each \(\mathscr T_n(x)\) is nonempty by
\eqref{eq:BT-cylinder-union}. Moreover, if a cylinder corresponding to a
path of length \(n+1\) contains \(x\), then the cylinder corresponding to
its length-\(n\) prefix also contains \(x\). Thus these paths form an
infinite finitely branching tree. By K\"onig's lemma, there is an infinite
admissible path
\[
        \omega=e_1e_2e_3\cdots
\]
such that \(x\) lies in every prefix cylinder. Every edge map has
contraction ratio \(\alpha\), so the diameter of the length-\(n\) prefix
cylinder is at most \(\alpha^n\). The nested prefix cylinders therefore
have a singleton intersection.

The preceding first-return calculation and the induction proving
\eqref{eq:BT-cylinder-union} show internally that the graph cylinders are
exactly the closed stationary inducing cylinders. The symbolic result of
Bruin--Troubetzkoy is used only to relate infinite admissible paths to
points of the dynamical attractor.

Let \(x\in J_{\mathrm{dyn}}\cap\overline{I_i}\). If \(x\) lies
in the interior of \(I_i\), or if the half-open convention assigns the
common branch endpoint \(x\) to \(I_i\), then
Proposition~\ref{lem:BT-symbolic-coding} shows that the future itinerary
\(\kappa(x)\) belongs to \(X_\chi\). Its inducing prefix of depth \(n\)
determines one of the closed graph cylinders comprising \(E_i^{(n)}\),
and hence
\[
        x\in E_i^{(n)}
        \qquad
        \text{for every }n\geq0.
\]

It remains only to explain the adjacent-branch interpretation at a shared
endpoint. At depth one, the explicit closed ranges in
\eqref{eq:BT-first-level-ranges} show that adjacent branch cylinders contain
their common endpoint. At every subsequent depth, the stationary refinement
replaces a closed parent cylinder by closed affine images of the first-level
cylinders. Therefore the same common endpoint remains in the closed child
cylinder on each adjacent side. Inductively, if a branch endpoint belongs
to \(J_{\mathrm{att}}\) and lies in both \(\overline{I_i}\) and
\(\overline{I_{i+1}}\), then it belongs to both \(E_i^{(n)}\) and
\(E_{i+1}^{(n)}\) for every \(n\), independently of which branch is selected
by the half-open definition of \(S\). This removes the only ambiguity caused
by shared branch endpoints.

Consequently,
\[
        J_{\mathrm{dyn}}\cap\overline{I_i}
        \subset
        \bigcap_{n\geq0}E_i^{(n)}.
\]
The same inclusion holds for the compactification endpoint \(1\), when it
belongs to \(J_{\mathrm{att}}\), by the left-continuous compactified coding and the closed
cylinder convention. Therefore
\[
        J_{\mathrm{att}}\cap\overline{I_i}
        \subset
        \bigcap_{n\geq0}E_i^{(n)}.
\]

Conversely, suppose that
\[
        x\in\bigcap_{n\geq0}E_i^{(n)}.
\]
The finitely branching cylinder-tree argument above produces an infinite
admissible graph path whose nested closed cylinders have intersection
\(\{x\}\). The corresponding symbolic sequence belongs to \(X_\chi\).
By the surjectivity part of the Bruin--Troubetzkoy symbolic identification,
in the form stated in Proposition~\ref{lem:BT-symbolic-coding}, this sequence is
the itinerary of at least one point \(y\) of the compactified attractor.
Both \(x\) and \(y\) belong to every prefix cylinder. Since the diameters of
these cylinders tend to zero,
\[
        x=y.
\]
Hence
\[
        x\in J_{\mathrm{att}}\cap\overline{I_i}.
\]
Therefore
\[
        K_i=J_{\mathrm{att}}\cap\overline{I_i}=\widehat J_i.
\]

Finally, on the product of triples of nonempty compact sets, equipped with
the maximum Hausdorff metric,
\[
        d_{\max}(\mathcal F(E),\mathcal F(F))
        \leq
        \alpha\,d_{\max}(E,F).
\]
Thus \(\mathcal F\) is a strict contraction and has a unique nonempty
compact fixed triple. Since
\[
        K_i=\bigcap_{n\geq0}E_i^{(n)}
\]
and \(\mathcal E^{(n+1)}=\mathcal F(\mathcal E^{(n)})\), the triple
\[
        (K_1,K_2,K_3)
        =
        (\widehat J_1,\widehat J_2,\widehat J_3)
\]
is that fixed point. Substitution into the definition of \(\mathcal F\)
gives the exact compact-set equations
\eqref{eq:BT-exact-graph-directed}.
\end{proof}

The graph-directed system satisfies the strong open set condition. Take
\[
        O_1=(0,s),\qquad
        O_2=(s,1-\beta),\qquad
        O_3=(1-\beta,1).
\]
The edge images are
\[
\begin{aligned}
L(O_2)&=(0,\beta s),&
L(O_3)&=(\beta s,\beta),&
M(O_2)&=(\alpha,s),\\
\phi(O_1)&=(s,1-\beta),&
\phi(O_2)&=(1-\beta,1-\alpha\beta),&
\phi(O_3)&=(1-\alpha\beta,1).
\end{aligned}
\]
They lie in the appropriate vertex open sets and are pairwise disjoint at each vertex.

It remains to check that every \(O_i\) meets its compact attractor piece. The admissible
cycle from vertex \(1\) through vertex \(2\) and back has contraction
\[
        L\circ\phi(x)=\beta x.
\]
Its fixed point \(0\) belongs to \(\widehat J_1\). Hence
\[
        s=\phi(0)\in\widehat J_2,
        \qquad
        \alpha=M(s)\in\widehat J_1.
\]
Moreover,
\[
        s+\beta=\phi(\alpha)\in\widehat J_2
\]
and
\[
        1-\beta+\alpha\beta=\phi^2(\alpha)\in\widehat J_3.
\]
Since
\[
        \alpha\in O_1,\qquad
        s+\beta\in O_2,\qquad
        1-\beta+\alpha\beta\in O_3,
\]
we have
\[
        O_i\cap\widehat J_i\neq\emptyset
        \qquad(i=1,2,3).
\]
Thus the open set condition is strong.

All graph maps have contraction ratio \(\alpha\), and their incidence matrix is the matrix
\(A\) in \eqref{eq:BT-matrix}. Let \(d>0\) be determined by
\[
        \alpha^d=s.
\]
Since the Perron--Frobenius eigenvalue of \(A\) is \(1/s\),
\[
        \rho(\alpha^dA)=1.
\]
The exact compact-set equations and the strong open set condition allow the
Mauldin--Williams graph-directed similarity theorem
\cite{MauldinWilliams1988} to be applied directly. It gives
\[
        \dim_H\widehat J_i=d,
        \qquad
        0<\mathcal H^d(\widehat J_i)<\infty
        \quad(i=1,2,3).
\]
The same dimension conclusion is also consistent with the published estimates of
Edgar--Golds \cite[Theorems~3.5 and~3.14]{EdgarGolds1999}.

Let
\[
        m_i=\mathcal H^d(\widehat J_i),
        \qquad
        m=(m_1,m_2,m_3)^{\mathsf T}.
\]
The pieces entering a fixed vertex have disjoint relative interiors; their intersections are
at most common endpoints, which have zero \(\mathcal H^d\)-measure because \(d>0\).
Therefore the exact graph equations imply
\[
        m=\alpha^dAm=sAm.
\]
Perron--Frobenius uniqueness identifies the normalized vector with the frequency vector
\(p\) in \eqref{eq:BT-frequency}.

We now identify the invariant probability measure with normalized
\(d\)-dimensional Hausdorff measure. This argument uses the exact graph
cylinders established above and does not require a separate invariance
statement from \cite{BruinTroubetzkoy2003}.

Put
\[
        M:=\mathcal H^d(J_{\mathrm{att}})
\]
and define
\[
        \nu
        :=
        \frac{\mathcal H^d|_{J_{\mathrm{att}}}}{M}.
\]
Since the compact pieces \(\widehat J_i\) overlap only at endpoints and
\(d>0\),
\[
        M
        =
        \sum_{i=1}^{3}\mathcal H^d(\widehat J_i).
\]
Consequently,
\[
        p_i
        =
        \frac{\mathcal H^d(\widehat J_i)}{M},
        \qquad i=1,2,3.
\tag{B.3}
\label{eq:BT-Hausdorff-vertex-masses}
\]

Let
\[
        \gamma=e_1e_2\cdots e_n
\]
be an admissible directed path in the graph of
Lemma~\ref{lem:BT-exact-graph-directed}. If \(i(\gamma)\) and
\(t(\gamma)\) denote its initial and terminal vertices, put
\[
        C_\gamma
        :=
        f_\gamma(\widehat J_{t(\gamma)})
        \subset\widehat J_{i(\gamma)}.
\]
Every edge map has similarity ratio \(\alpha\). Hence
\(f_\gamma\) has similarity ratio \(\alpha^n\), and therefore
\[
\begin{aligned}
        \nu(C_\gamma)
        &=
        \frac{
        \mathcal H^d
        \bigl(f_\gamma(\widehat J_{t(\gamma)})\bigr)
        }{M}\\
        &=
        \alpha^{nd}
        \frac{\mathcal H^d(\widehat J_{t(\gamma)})}{M}\\
        &=
        s^n p_{t(\gamma)}.
\end{aligned}
\tag{B.4}
\label{eq:BT-Hausdorff-cylinder-mass}
\]
Here we used \(\alpha^d=s\).

Let \(\Omega_A\) be the infinite path space of the stationary directed
multigraph with incidence matrix \(A\). For a finite admissible path
\(\gamma\) of length \(n\), denote its path cylinder by
\[
        [\gamma]
        :=
        \{
        \omega\in\Omega_A:
        \omega_1\cdots\omega_n=\gamma
        \}.
\]
The Perron--Frobenius probability measure \(\mathbb P\) on \(\Omega_A\)
is determined by
\[
        \mathbb P([\gamma])
        =
        s^n p_{t(\gamma)}.
\tag{B.5}
\label{eq:BT-PF-path-cylinder}
\]
These values are compatible under refinement. Indeed, if \(\gamma\)
terminates at vertex \(j\), then its one-edge extensions are indexed by
the edges from \(j\) to \(k\), and
\[
\begin{aligned}
\sum_{\substack{e\\i(e)=j}}
\mathbb P([\gamma e])
&=
\sum_{k=1}^{3}A_{jk}s^{n+1}p_k\\
&=
s^{n+1}(Ap)_j\\
&=
s^{n+1}\frac{p_j}{s}\\
&=
s^np_j\\
&=
\mathbb P([\gamma]).
\end{aligned}
\]
Thus \eqref{eq:BT-PF-path-cylinder} defines a probability measure on
the infinite path space. For the substitution ordering specified in
\eqref{eq:BT-substitution-edge-order} below, it is the unique invariant
probability measure of the stationary Bratteli--Vershik system associated
with \(\chi\); see
\cite[Theorem~3.8 and Remark~3.9]{BezuglyiKwiatkowskiMedynetsSolomyak2010}.
The correspondence with the invariant probability of the substitution
system is given by
\cite[Theorem~5.4]{BezuglyiKwiatkowskiMedynetsSolomyak2010}.

For
\[
        \omega=e_1e_2e_3\cdots\in\Omega_A,
\]
define
\[
        \pi(\omega)
        :=
        \bigcap_{n\geq1}
        f_{e_1}\circ\cdots\circ f_{e_n}
        \bigl(
        \widehat J_{t(e_n)}
        \bigr).
\]
Since every prefix cylinder has diameter at most \(\alpha^n\), the
intersection consists of a single point. Thus
\[
        \pi:\Omega_A\longrightarrow J_{\mathrm{att}}
\]
is well-defined and continuous.

By Lemma~\ref{lem:BT-exact-graph-directed}, the sets
\(\pi([\gamma])\) are exactly the closed graph cylinders \(C_\gamma\).
The only possible failure of injectivity occurs when a point is a common
endpoint of two adjacent cylinders. Let \(E_{\mathrm{cyl}}\) denote the
union of all common endpoints of cylinders of all depths. Because there
are only finitely many cylinders at each depth, \(E_{\mathrm{cyl}}\) is
countable.

The measure \(\mathbb P\) is non-atomic. Indeed, for every
\(\omega\in\Omega_A\),
\[
\begin{aligned}
        \mathbb P(\{\omega\})
        &\leq
        \mathbb P
        \bigl(
        [\omega_1\cdots\omega_n]
        \bigr)\\
        &=
        s^n p_{t(\omega_1\cdots\omega_n)}
        \leq
        s^n\max_i p_i
        \longrightarrow0.
\end{aligned}
\]
Similarly, \(\nu\) is non-atomic because \(d>0\). Moreover, the relative
interiors of the depth-\(n\) cylinders are pairwise disjoint inside each
of the three branch closures. Hence there is a uniform finite bound on
the number of depth-\(n\) cylinders containing any fixed endpoint. It
follows that, for every \(x\in E_{\mathrm{cyl}}\), the set
\(\pi^{-1}(\{x\})\) is covered by a uniformly bounded number of
length-\(n\) path cylinders, and therefore
\[
        \mathbb P\bigl(\pi^{-1}(\{x\})\bigr)
        \leq
        C s^n\max_i p_i
        \longrightarrow0
\]
for a constant \(C\) independent of \(n\). Since
\(E_{\mathrm{cyl}}\) is countable,
\[
        \mathbb P\bigl(\pi^{-1}(E_{\mathrm{cyl}})\bigr)=0,
        \qquad
        \nu(E_{\mathrm{cyl}})=0.
\tag{B.6}
\label{eq:BT-endpoint-null}
\]

For every graph cylinder \(C_\gamma\), equations
\eqref{eq:BT-Hausdorff-cylinder-mass} and
\eqref{eq:BT-PF-path-cylinder} give
\[
        \nu(C_\gamma)
        =
        \mathbb P([\gamma]).
\tag{B.7}
\label{eq:BT-measure-cylinder-equality}
\]
To avoid ambiguity at common endpoints, remove
\(E_{\mathrm{cyl}}\) and assign every remaining point to its unique
depth-\(n\) cylinder. This gives, for each \(n\), a finite Borel
partition \(\mathcal Q_n\) of \(J_{\mathrm{att}}\setminus E_{\mathrm{cyl}}\). The
partitions are nested, and the diameter of every atom of \(\mathcal Q_n\)
is at most \(\alpha^n\). Consequently,
\[
        \sigma\left(
        \bigcup_{n\geq1}\mathcal Q_n
        \right)
        =
        \mathcal B(J_{\mathrm{att}}\setminus E_{\mathrm{cyl}}).
\]
Equation \eqref{eq:BT-measure-cylinder-equality} shows that
\(\pi_*\mathbb P\) and \(\nu\) agree on every atom of every
\(\mathcal Q_n\). By the monotone-class theorem, they agree on all Borel
subsets of \(J_{\mathrm{att}}\setminus E_{\mathrm{cyl}}\). Both measures assign zero
mass to \(E_{\mathrm{cyl}}\), so
\[
        \pi_*\mathbb P=\nu
        \qquad\text{on }J_{\mathrm{att}}.
\tag{B.8}
\label{eq:BT-Hausdorff-path-pushforward}
\]

\smallskip
\noindent
\textbf{Bratteli--Vershik intertwining and invariance.}
We now verify explicitly that \(\pi_*\mathbb P\) is invariant under the
original interval translation map.

Order the edges having the same terminal vertex according to the order
in which the corresponding letters occur in the return words
\[
        \chi(1)=2,
        \qquad
        \chi(2)=311,
        \qquad
        \chi(3)=31.
\]
Thus the substitution order is
\[
\begin{array}{c|c}
\text{terminal vertex}&\text{ordered incoming edges}\\
\hline
1&2\longrightarrow1\;(\phi)\\
2&3\longrightarrow2\;(\phi)
  \;<\;
  1\longrightarrow2\;(L)
  \;<\;
  1\longrightarrow2\;(M)\\
3&3\longrightarrow3\;(\phi)
  \;<\;
  1\longrightarrow3\;(L).
\end{array}
\tag{B.9}
\label{eq:BT-substitution-edge-order}
\]
Let \(\Omega_{\max}\) and \(\Omega_{\min}\) be the sets of paths all of
whose edges are respectively maximal or minimal in this order. These
sets are finite. The associated Vershik successor is the Borel
bijection
\[
        V:
        \Omega_A\setminus\Omega_{\max}
        \longrightarrow
        \Omega_A\setminus\Omega_{\min}
\]
obtained by replacing the first non-maximal edge by its successor and
resetting the preceding finite prefix to the corresponding minimal
path. The measure \(\mathbb P\) is invariant under this successor map;
equivalently, it is invariant under the tail equivalence relation of
the stationary ordered Bratteli diagram
\cite[Theorem~3.8 and Remark~3.9]{BezuglyiKwiatkowskiMedynetsSolomyak2010}.
Since \(\mathbb P\) is non-atomic, both extremal path sets have
\(\mathbb P\)-measure zero.

We next remove the geometric points at which a half-open branch choice
or a cylinder boundary can create an ambiguity. Put
\[
        H=\{0,1-\alpha,1-\beta\}
\]
and define the countable set
\[
\begin{aligned}
        D_{\mathrm{geom}}
        :=\;&\{1\}
        \cup
        \bigcup_{r\geq0}
        S^r\bigl((E_{\mathrm{cyl}}\cup H)\cap J_{\mathrm{dyn}}\bigr)\\
        &\cup
        \bigcup_{r\geq0}
        S^{-r}\bigl((E_{\mathrm{cyl}}\cup H)\cap J_{\mathrm{dyn}}\bigr),
\end{aligned}
\tag{B.10}
\label{eq:BT-geometric-exceptional-set}
\]
where \(S^{-r}\) denotes the full inverse image. The set is countable:
\(E_{\mathrm{cyl}}\cup H\) is countable, and every point has only
finitely many preimages under each iterate of a finite-branch interval
translation map. Since \(\nu\) is non-atomic, equation
\eqref{eq:BT-Hausdorff-path-pushforward} gives
\[
        \mathbb P\bigl(\pi^{-1}(D_{\mathrm{geom}})\bigr)
        =
        \nu(D_{\mathrm{geom}})
        =0.
\tag{B.11}
\label{eq:BT-geometric-exceptional-null}
\]

We claim that
\[
        \pi\circ V
        =
        S\circ\pi
        \qquad
        \mathbb P\text{-almost everywhere}.
\tag{B.12}
\label{eq:BT-Vershik-intertwining}
\]
Indeed, the depth-one edge order in
\eqref{eq:BT-substitution-edge-order} is exactly the temporal order of
the levels in the three first-return towers with return words
\(2,311,31\). The cylinder identification in
Lemma~\ref{lem:BT-exact-graph-directed} then implies inductively that,
at every depth \(n\), the lexicographically ordered paths ending at a
fixed depth-\(n\) base enumerate the corresponding closed tower levels
in their actual \(S\)-orbit order. Consequently, if a finite path is
not the top level of its depth-\(n\) tower, \(S\) maps the relative
interior of its geometric cylinder by translation onto the relative
interior of the successor cylinder.

Now take
\[
        \omega\in
        \Omega_A\setminus
        \bigl(
        \Omega_{\max}
        \cup
        \pi^{-1}(D_{\mathrm{geom}})
        \bigr).
\]
Let \(r\) be the first level at which the edge of \(\omega\) is not
maximal. The definition of \(V\) performs the usual carry at level
\(r\). For every \(n\geq r\), the preceding tower-order observation
therefore gives
\[
        S\bigl(\pi(\omega)\bigr)
        \in
        C_{(V\omega)|n},
\]
where \(C_{(V\omega)|n}\) is the geometric cylinder determined by the
first \(n\) edges of \(V\omega\). The exclusion of
\(D_{\mathrm{geom}}\) guarantees that no branch or cylinder endpoint is
encountered, so the actual half-open map \(S\), rather than merely a
one-sided affine continuation, is used in this relation. The cylinders
\(C_{(V\omega)|n}\) are nested and have diameter at most \(\alpha^n\).
Their intersection is therefore the singleton \(\{\pi(V\omega)\}\),
which proves
\eqref{eq:BT-Vershik-intertwining}.

Finally, let \(B\subset J_{\mathrm{dyn}}\) be Borel. Using
\eqref{eq:BT-Vershik-intertwining}, the \(V\)-invariance of
\(\mathbb P\), and the nullity of the exceptional sets, we obtain
\[
\begin{aligned}
(\pi_*\mathbb P)(S^{-1}B)
&=
\mathbb P\bigl(\pi^{-1}(S^{-1}B)\bigr)\\
&=
\mathbb P\bigl((\pi\circ V)^{-1}(B)\bigr)\\
&=
\mathbb P\bigl(V^{-1}(\pi^{-1}(B))\bigr)\\
&=
\mathbb P\bigl(\pi^{-1}(B)\bigr)\\
&=
(\pi_*\mathbb P)(B).
\end{aligned}
\tag{B.13}
\label{eq:BT-path-pushforward-invariance}
\]
Also,
\[
        (\pi_*\mathbb P)(\{1\})
        =
        \nu(\{1\})
        =0,
\]
so \(\pi_*\mathbb P\) is an \(S\)-invariant probability on the actual
dynamical space \(J_{\mathrm{dyn}}\). By the uniqueness proved in the
preceding subsection,
\[
        \pi_*\mathbb P=\mu.
\]
Together with \eqref{eq:BT-Hausdorff-path-pushforward}, this gives
\[
        \boxed{
        \mu
        =
        \frac{\mathcal H^d|_{J_{\mathrm{att}}}}{\mathcal H^d(J_{\mathrm{att}})}.
        }
\tag{B.14}
\label{eq:BT-Hausdorff-measure}
\]
In particular, normalized Hausdorff measure is \(S\)-invariant on
\(J_{\mathrm{dyn}}\). When it is regarded as a measure on the compact
attractor \(J_{\mathrm{att}}\), the added endpoint \(1\) has zero mass. The support of
\(\mu\), viewed in the compactification, is exactly \(J_{\mathrm{att}}\).

The resulting identification with normalized Hausdorff measure is
consistent with the Hausdorff-dimension and ergodic results obtained for
this family in \cite{BruinTroubetzkoy2003}. The argument above is included
to make the invariance and the precise normalization used here explicit.

We can now prove branchwise nonsingularity without any invertibility
assumption on the symbolic shift. Let
\[
        N\subset J_{\mathrm{dyn}}\cap I_k
\]
be Borel and suppose that
\[
        \mu(N)=0.
\]
By \eqref{eq:BT-Hausdorff-measure},
\[
        \mathcal H^d(N)=0.
\]
The branch map
\[
        \tau_k(x)=x+c_k
\]
is an isometry, and therefore
\[
        \mathcal H^d(\tau_k(N))
        =
        \mathcal H^d(N)
        =
        0.
\]
Since \(\mu\) is supported on \(J_{\mathrm{att}}\),
\[
\begin{aligned}
        \mu(\tau_k(N))
        &=
        \mu(\tau_k(N)\cap J_{\mathrm{att}})\\
        &=
        \frac{
        \mathcal H^d(\tau_k(N)\cap J_{\mathrm{att}})
        }{
        \mathcal H^d(J_{\mathrm{att}})
        }\\
        &=0.
\end{aligned}
\]
Thus \(\mu\) is branchwise nonsingular on the support.

\section{Complete renormalization-tower bookkeeping}
\label{app:gap-tower-details}

This appendix contains the full proof of the renormalization gap-tower lemma,
including the exact compact refinement, the limiting tower support, the gap
endpoints, and the exhaustion argument.

\begin{proof}[Proof of the renormalization gap-tower lemma]
We separate the combinatorial refinement into explicit steps.

\smallskip
\noindent
\textbf{Step 1: return words.}
Let \(F_m\) be the first-return map to \(\Delta_m\). Stationary inducing gives
\[
        F_m=\phi^m\circ S\circ\phi^{-m}.
\tag{C.1}
\label{eq:BT-induced-conjugacy}
\]
For \(m=0\), the return itinerary of \(\Delta_{0,i}=I_i\) is the one-letter word
\(i=\chi^0(i)\). Suppose that the return itinerary of every \(\Delta_{m,a}\) is
\(\chi^m(a)\). Write
\[
        \chi(i)=a_0a_1\cdots a_{t-1}.
\]
At level \(m+1\), each \(F_m\)-step through the piece \(\Delta_{m,a_q}\) expands into the
\(S\)-word \(\chi^m(a_q)\). Hence the complete return word is
\[
        \chi^m(a_0)\chi^m(a_1)\cdots\chi^m(a_{t-1})
        =
        \chi^m(\chi(i))
        =
        \chi^{m+1}(i).
\]
Its length is \(r_{m+1}(i)=|\chi^{m+1}(i)|\). This proves \textup{(i)}.

\smallskip
\noindent
\textbf{Step 2: refinement inside the base.}
Transporting the first-level decomposition \eqref{eq:BT-first-tower} by \(\phi^m\) and using
\eqref{eq:BT-induced-conjugacy} gives
\[
\Delta_m\setminus B_m
\doteq
\bigcup_{k=1}^{3}
\bigcup_{0\leq q<|\chi(k)|}
F_m^q(\Delta_{m+1,k}).
\tag{C.2}
\label{eq:BT-base-refinement}
\]
The interiors of the sets on the right are pairwise disjoint.

Fix \(k\) and again write \(\chi(k)=a_0\cdots a_{t-1}\). Define the cumulative
\(S\)-times
\[
        R_q:=\sum_{p=0}^{q-1}r_m(a_p),
        \qquad R_0=0.
\]
On \(\Delta_{m+1,k}\),
\[
        F_m^q=S^{R_q},
        \qquad
        F_m^q(\Delta_{m+1,k})\subset\Delta_{m,a_q},
\]
away from the finite endpoint set. Therefore
\[
\operatorname{Sat}^{\mathrm{aff}}_m
\bigl(F_m^q(\Delta_{m+1,k})\bigr)
\doteq
\bigcup_{j=0}^{r_m(a_q)-1}
S^{R_q+j}(\Delta_{m+1,k}).
\]
The integer intervals
\[
        [R_q,R_q+r_m(a_q)-1],
        \qquad 0\leq q<t,
\]
are consecutive and disjoint, and their union is
\[
        \{0,1,\ldots,r_{m+1}(k)-1\},
\]
because
\[
\sum_{q=0}^{t-1}r_m(a_q)
=
|\chi^m(\chi(k))|
=
|\chi^{m+1}(k)|
=
r_{m+1}(k).
\]
Taking the union over \(q\) and then over \(k\) in
\eqref{eq:BT-base-refinement} proves
\[
        \operatorname{Sat}^{\mathrm{aff}}_m(\Delta_m\setminus B_m)
        \doteq U_{m+1}.
\tag{C.3}
\label{eq:BT-retained-saturation}
\]

\smallskip
\noindent
\textbf{Step 3: the removed subtower.}
Since \(B_m\subset\Delta_{m,1}^{\circ}\), the prescribed affine branch and the
actual iterate agree on \(B_m\). Therefore
\[
\begin{aligned}
\operatorname{Sat}^{\mathrm{aff}}_m(B_m)
&=
\bigcup_{j=0}^{r_m(1)-1}
\widetilde S_{m,1,j}(B_m)\\
&=
\bigcup_{j=0}^{r_m(1)-1}S^j(B_m)\\
&=
\bigsqcup_{j=0}^{\ell_m-1}G_{m,j}.
\end{aligned}
\]
The first-level tower has disjoint interiors. At every later stage, refinement takes place
independently inside an old tower level, where the relevant affine branch is a translation.
Induction therefore shows that all level-\(m\) tower intervals have pairwise disjoint
interiors. Together with \eqref{eq:BT-retained-saturation}, this proves \textup{(ii)}.

\smallskip
\noindent
\textbf{Step 4: exact compact refinement.}

Put
\[
        R_m
        :=
        \overline{\Delta}_m\setminus B_m.
\]
Since \(B_m\) is open, \(R_m\) is compact and contains both endpoints
of \(B_m\). The exact closed-cylinder refinement established above gives
\[
        K_{m+1}
        =
        \widehat{\operatorname{Sat}}_m(R_m).
\tag{C.4}
\label{eq:BT-compact-retained-saturation}
\]

Since
\[
        \overline{\Delta}_m
        =
        R_m\cup\overline{B_m},
\]
we obtain
\[
\begin{aligned}
        K_m
        &=
        \widehat{\operatorname{Sat}}_m
        (\overline{\Delta}_m)\\
        &=
        \widehat{\operatorname{Sat}}_m(R_m)
        \cup
        \widehat{\operatorname{Sat}}_m(\overline{B_m})\\
        &=
        K_{m+1}
        \cup
        \widehat{\operatorname{Sat}}_m(\overline{B_m}).
\end{aligned}
\]
Because
\[
        \overline{B_m}\subset\Delta_{m,1}^{\circ},
\]
the last compact saturation is
\[
\begin{aligned}
        \widehat{\operatorname{Sat}}_m(\overline{B_m})
        &=
        \bigcup_{j=0}^{r_m(1)-1}
        \widetilde S_{m,1,j}(\overline{B_m})\\
        &=
        \bigcup_{j=0}^{\ell_m-1}
        \overline{G_{m,j}}.
\end{aligned}
\]
Thus
\[
        K_m
        =
        K_{m+1}
        \cup
        \bigcup_{j=0}^{\ell_m-1}
        \overline{G_{m,j}}.
\tag{C.5}
\label{eq:BT-closed-tower-refinement}
\]

The endpoints of \(B_m\) belong to \(R_m\). Applying the prescribed
affine tower maps therefore gives
\[
        \partial G_{m,j}
        \subset K_{m+1}
        \qquad
        (0\leq j<\ell_m).
\tag{C.6}
\label{eq:BT-gap-boundaries-retained}
\]
On the other hand, refinement takes place independently inside the
pairwise interior-disjoint old tower levels. Inside the level
corresponding to \(\widetilde S_{m,1,j}\), the retained part is
\[
        \widetilde S_{m,1,j}(R_m\cap\overline{\Delta}_{m,1}),
\]
while the removed part is the open interval
\[
        G_{m,j}
        =
        \widetilde S_{m,1,j}(B_m).
\]
These sets are disjoint. Every other retained closed level can meet this
old level only at a tower endpoint. Hence
\[
        G_{m,j}\cap K_{m+1}=\varnothing.
\]
It follows that
\[
        K_m\setminus K_{m+1}
        =
        \bigsqcup_{j=0}^{\ell_m-1}G_{m,j}.
\tag{C.7}
\label{eq:BT-exact-tower-refinement}
\]
In particular,
\[
        K_{m+1}\subset K_m,
\]
and the gap towers removed at distinct refinement levels are pairwise
disjoint.

\smallskip
\noindent
\textbf{Identification of the limiting tower support.}
We now prove explicitly that
\[
        J=\bigcap_{m\geq0}K_m.
\tag{C.8}
\label{eq:BT-J-intersection-Km}
\]

Recall from Lemma~\ref{lem:BT-exact-graph-directed} that
\(E_i^{(m)}\) is the union of all closed depth-\(m\) stationary inducing
cylinders lying in the branch closure \(\overline{I_i}\), and that
\[
        \widehat J_i
        =
        J\cap\overline{I_i}
        =
        \bigcap_{m\geq0}E_i^{(m)}.
\tag{C.9}
\label{eq:BT-Ji-cylinder-intersection}
\]
The compact tower support \(K_m\) is the union of all closed
depth-\(m\) inducing cylinders, including all their tower levels.
Consequently,
\[
        K_m
        =
        E_1^{(m)}
        \cup
        E_2^{(m)}
        \cup
        E_3^{(m)}.
\tag{C.10}
\label{eq:BT-Km-branch-cylinder-union}
\]
This is an exact equality. A common branch endpoint may belong to two
adjacent sets \(E_i^{(m)}\), but this causes no ambiguity in their union.

We first prove
\[
        J\subset\bigcap_{m\geq0}K_m.
\]
Let \(x\in J\). Since
\[
        [0,1]
        =
        \overline{I_1}
        \cup
        \overline{I_2}
        \cup
        \overline{I_3},
\]
there is at least one \(i\in\{1,2,3\}\) such that
\[
        x\in J\cap\overline{I_i}
        =
        \widehat J_i.
\]
By \eqref{eq:BT-Ji-cylinder-intersection},
\[
        x\in E_i^{(m)}
        \subset K_m
\]
for every \(m\geq0\). Hence
\[
        x\in\bigcap_{m\geq0}K_m.
\]

Conversely, let
\[
        x\in\bigcap_{m\geq0}K_m.
\]
For every \(m\geq0\), define
\[
        A_m(x)
        :=
        \bigl\{
        i\in\{1,2,3\}:
        x\in E_i^{(m)}
        \bigr\}.
\]
By \eqref{eq:BT-Km-branch-cylinder-union}, the set \(A_m(x)\) is
nonempty. Moreover,
\[
        E_i^{(m+1)}\subset E_i^{(m)}
\]
for every \(i\), and therefore
\[
        A_{m+1}(x)\subset A_m(x).
\]
Thus
\[
        A_0(x)\supset A_1(x)\supset A_2(x)\supset\cdots
\]
is a decreasing sequence of nonempty subsets of the finite set
\(\{1,2,3\}\). It follows that
\[
        \bigcap_{m\geq0}A_m(x)\neq\varnothing.
\]
Choose
\[
        i\in\bigcap_{m\geq0}A_m(x).
\]
Then
\[
        x\in E_i^{(m)}
        \qquad\text{for every }m\geq0,
\]
and hence, again by \eqref{eq:BT-Ji-cylinder-intersection},
\[
        x\in
        \bigcap_{m\geq0}E_i^{(m)}
        =
        \widehat J_i
        \subset J.
\]
This proves the reverse inclusion and establishes
\eqref{eq:BT-J-intersection-Km}.

In particular,
\[
        J\subset K_{m+1}
        \qquad\text{for every }m\geq0.
\]
Combining this with the exact refinement identity
\eqref{eq:BT-exact-tower-refinement}, we obtain
\[
        G_{m,j}\cap J=\varnothing
        \qquad
        \text{for all }
        m\geq0,\quad 0\leq j<\ell_m.
\]

\smallskip
\noindent
\textbf{Step 5: gap endpoints.}
The exact graph-directed equations contain a cycle on \(\widehat J_1\) whose composition is
\(x\mapsto\beta x\), and a self-loop on \(\widehat J_3\) given by
\(x\mapsto s+\alpha x\). Their fixed points give
\[
        0\in\widehat J_1,
        \qquad
        1\in\widehat J_3.
\]
Consequently,
\[
        s=\phi(0)\in\widehat J_2,
        \qquad
        \beta=L(1)\in\widehat J_1,
        \qquad
        \alpha=M(s)\in\widehat J_1.
\]
Thus both endpoints of \(G_0=(\beta,\alpha)\) belong to \(J\). The exact graph equations
also give \(\phi(J)\subset J\), so the endpoints of \(B_m=\phi^m(G_0)\) lie in \(J\).
Following a prescribed return word applies the continuous graph-cylinder maps defining the
closed tower levels; these maps send attractor endpoints to attractor points. Hence every
endpoint of every \(G_{m,j}\) belongs to \(J\).

Since \(G_{m,j}\cap J=\emptyset\) and both endpoints lie in \(J\), each \(G_{m,j}\) is a
complete connected component of \([0,1]\setminus J\).

\smallskip
\noindent
\textbf{Step 6: exhaustion.}
Every iterate before return is a translation on \(B_m\), so
\[
        |G_{m,j}|=\alpha^m(\alpha-\beta).
\tag{C.11}
\label{eq:BT-gap-length}
\]
The tower heights satisfy
\[
        \ell_0=1,\qquad
        \ell_1=1,\qquad
        \ell_2=3,
\]
and the characteristic polynomial of \(A\) gives
\[
        \ell_{m+3}
        =
        \ell_{m+2}+2\ell_{m+1}-\ell_m.
\tag{C.12}
\label{eq:BT-height-recurrence}
\]
Therefore
\[
        \sum_{m\geq0}\ell_mz^m
        =
        \frac{1}{1-z-2z^2+z^3}.
\tag{C.13}
\label{eq:BT-height-generating}
\]
The Perron--Frobenius growth rate of \(\ell_m\) is \(1/s\), and
\(\alpha/s<1\), so the series converges at \(z=\alpha\). Using
\(P(\alpha)=0\),
\[
\begin{aligned}
\sum_{m\geq0}\sum_{j=0}^{\ell_m-1}|G_{m,j}|
&=
(\alpha-\beta)\sum_{m\geq0}\ell_m\alpha^m\\
&=
\frac{\alpha-\beta}{1-\alpha-2\alpha^2+\alpha^3}\\
&=
1.
\end{aligned}
\tag{C.14}
\label{eq:BT-gap-total-length}
\]
The displayed intervals are already pairwise disjoint connected components. Any additional
component of \([0,1]\setminus J\) would be a nonempty relatively open interval and would have
positive length, contradicting \eqref{eq:BT-gap-total-length}. Thus the displayed gap towers
exhaust all complementary components, proving \textup{(iv)}.
\end{proof}

\end{document}